\documentclass{conm-p-l}

\usepackage{url}
\usepackage{pgf}
\usepackage{tikz}
\usepackage{tikz-cd}
\usepackage{pgfplots}
\usepackage[colorlinks]{hyperref}
\usepackage{framed}
\usepackage[british]{babel}
\usepackage{amsthm}
\usepackage{amsmath}
\usepackage{amsfonts}
\usepackage{amssymb}
\usepackage{mathtools}
\usepackage{enumitem}
\usepackage{amscd}
\usepackage{graphicx}
\usepackage{mathrsfs}
\usepackage{extarrows}
\usepackage{stmaryrd}
\usepackage{mathdots}
\usepackage{fancyhdr}
\usepackage{verbatim}
\usepackage{xifthen}
\usepackage{xcolor}
\usepackage{clipboard}
\usepackage{etoolbox}
\usepackage{verbatim}
\usepackage{todonotes}
\usepackage{adjustbox}
\usepackage[page]{appendix}

\newtheorem{theorem}{Theorem}[section]
\newtheorem{prop}[theorem]{Proposition}
\newtheorem{lemma}[theorem]{Lemma}
\newtheorem{cor}[theorem]{Corollary}

\theoremstyle{definition}
\newtheorem{definition}[theorem]{Definition}
\newtheorem{example}[theorem]{Example}

\theoremstyle{remark}
\newtheorem{remark}[theorem]{Remark}

\numberwithin{equation}{section}

\newcommand{\N}{\mathbb{N}}
\newcommand{\R}{\mathbb{R}}

\newcommand{\PP}{\mathbb{P}}

\newcommand{\calc}{\mathscr{C}}
\newcommand{\cald}{\mathscr{D}}
\newcommand{\cale}{\mathscr{E}}

\newcommand{\calg}{\mathscr{G}}

\newcommand{\call}{\mathscr{L}}

\newcommand{\calp}{\mathcal{P}}

\newcommand{\calm}{\mathscr{M}}

\newcommand{\Top} {\mathsf{Top}}

\newcommand{\sSet} {\mathsf{sSet}}

\newcommand{\Lin}[1]{\mathsf{Lin}\br{#1}}

\newcommand{\Com}{\mathsf{Com}}
\newcommand{\Ass}{\mathsf{Ass}}
\newcommand{\Conf}{\mathsf{Conf}}
\newcommand{\NConf}{\mathsf{NConf}}
\newcommand{\FM}{\mathsf{FM}}
\newcommand{\MS}{\mathsf{MS}}

\newcommand{\Tot}[1]{\mathrm{Tot}\br{#1}}

\newcommand{\down}{\downarrow\!}

\newcommand{\op}{\mathrm{op}}

\newcommand{\colonLeftrightarrow}{\mathrel{\vcentcolon\Leftrightarrow}}

\newcommand{\limit}[1][]{ \ifthenelse{\isempty{#1}}{\mathrm{lim}}{\underset{#1}{\mathrm{lim}}}\, }
\newcommand{\colim}[1][]{
  \ifthenelse{\isempty{#1}}{\mathrm{colim}}{\underset{#1}{\mathrm{colim}}}\, }
\newcommand{\hocolim}[1][]{ \ifthenelse{\isempty{#1}}{\mathrm{hocolim}}{\underset{#1}{\mathrm{hocolim}}}\, }
\newcommand{\holim}[1][]{ \ifthenelse{\isempty{#1}}{\mathrm{holim}}{\underset{#1}{\mathrm{holim}}}\, }

\newcommand{\id}{\mathrm{id}}

\newcommand{\br}[1]{\left(#1\right)}
\newcommand{\cb}[1]{\left\{#1\right\}}
\newcommand{\sbr}[1]{\left[#1\right]}

\newcommand{\abs}[1]{\left\vert#1\right\vert}

\usepackage{scalerel}

\newcommand{\isoto}{\xlongrightarrow{\raisebox{-0.5em}{\smash{\ensuremath{\sim}}}}}

\newcommand{\isofrom}{\xlongleftarrow{\raisebox{-0.5em}{\smash{\ensuremath{\sim}}}}}

\tikzset{
  symbol/.style={
    draw=none,
    every to/.append style={
      edge node={node [sloped, allow upside down, auto={false}]{$#1$}}
    }
  }
}
\tikzset{
  iso/.style={
    every to/.append style={
      edge node={node [sloped, allow upside down, auto={false},
        font=\normalsize, yshift=2pt]{$\sim$}}
    }
  }
}
\tikzset{
  iso'/.style={
    every to/.append style={
      edge node={node [sloped, allow upside down, auto={false}, font=\normalsize, yshift=-4pt]{$\sim$}}
    }
  }
}

\begin{document}

\title[A small catalogue of $E_n$-operads]{A small catalogue of $E_n$-operads}


\author[A.~Beuckelmann]{Andr\'e Beuckelmann}
\address{Universit\"at Hamburg}
\email{andre.beuckelmann@uni-hamburg.de}

\author[I.~Moerdijk]{Ieke Moerdijk}
\address{Universiteit Utrecht}
\email{i.moerdijk@uu.nl}

\thanks{}

\date{\today}

\begin{abstract}
  In this largely expository paper, we will present a list of $E_n$-operads and give complete, and in some cases new, proofs of the equivalences between these operads. 
\end{abstract}

\maketitle

\begin{center}
{\em To Ezra Getzler, on the occasion of his 60th birthday.}  
\end{center}
\vspace{2em}

\section{Introduction}

The theory of operads started in topology with the analysis of (iterated) loop spaces by Stasheff~\cite{Sta63}, Boardman-Vogt~\cite{BV73} and May~\cite{May72}, in terms of actions on spaces by the little $n$-cubes operads $\calc_n$ for $1 < n <\infty$. An $E_n$-operad is by definition an operad that is (weakly) equivalent to $\calc_n$. These $E_n$-operads and their associated homology have played a central role in many parts of algebra, geometry and mathematical physics ever since, and Getzler has played an important part in the development of the theory (see for example, Getzer-Jones~\cite{GJ94}, Getzler-Kapranov~\cite{GK95} and~\cite{GK98}, Getzler’s Batalin-Vilkovisky paper~\cite{Get94}). We refer to~\cite{MSS02} for a brief early history of the theory of operads. More recently new $E_n$-operads of a more combinatorial nature have been constructed, for example in the context of the Deligne-conjecture~\cite{MS04} and of formality problems~\cite{Tam98}.

The goal of this paper is to give an exposition of a series of $E_n$-operads, and to provide complete proofs of the equivalence between all these. It is not our intention to give a complete list, but instead focus on the following operads:
\begin{enumerate}
\item The classical little $n$-cubes operad $\calc_n$ introduced in~\cite{May72} and~\cite{BV73}.
\item The complete graphs operad $\calg_n$ introduced by Berger~\cite{Ber97}, and its extension $\calg_n^{\mathrm{ext}}$ introduced by Brun, Fiedorowicz and Vogt~\cite{BFV04}.
\item The $n$-th filtration of the Barratt-Eccles operad~\cite{BE74} introduced by Smith~\cite{Smi89}, who conjectured it to be an $E_n$-operad, and studied by Kashiwabara~\cite{Kas93}. It will be denoted $\Gamma_n$ in this paper.
\item The Fulton-MacPherson compactification $\FM_n$ of the space of configurations of finitely many distinct points in $\R^n$.
\item The operad $\calm_n$ of Balteanu et al.\ defined in terms of $n$-fold monoidal categories, see~\cite{BFSV98}.
\item The McClure-Smith operad introduced in~\cite{MS04} and used in their proof of the Deligne conjecture. This operad will be denoted $\MS_n$ below.
\end{enumerate}

All these operads will be defined in detail below, and no prior knowledge of them is expected of the reader.

It is generally believed to be known that these operads are all equivalent, and hence all deserve the attribute ``$E_n$-operad''. However, the literature contains some proofs that seem to need some further clarification or contain gaps. In particular, there seems to be no correct proof in the literature that $\calg_n$ is actually an $E_n$-operad. We will be more explicit about this below (see Appendix~\ref{appB}). But to give just one example, in~\cite{Ber97}, Berger argues that $\calg_n$ is an $E_n$-operad, and~\cite{BFV04} remark that $\calg_n$ should really be replaced by what we denote $\calg_n^{\mathrm{ext}}$, but we have been unable to find a proof in the literature that the inclusion from $\calg_n$ into $\calg_n^{\mathrm{ext}}$ is an equivalence of operads. We will present a direct combinatorial proof of such an equivalence, which is one of the new elements in this paper. Besides this, we have tried to streamline some arguments in the literature, or replaced them by simpler ones. As another example, Salvatore~\cite{Sal01} has proved that the Fulton-MacPherson operad $\FM_n$ is an $E_n$-operad, by constructing a map into it from the Boardman-Vogt resolution of the little $n$-cubes operad. Instead, we will define a map from the Fulton-MacPherson operad into the Boardman-Vogt resolution $W\calg_n$ of the complete graph operad $\calg_n$, which provides a different and perhaps somewhat simpler proof that $\FM_n$ is an $E_n$-operad.

The contents of our paper, then, is as follows. After recalling some basic definitions in Section~\ref{secOperads}, we explain in the next section a general method for comparing operads in different categories, notably in spaces and in posets. Our method can be viewed as using a stratification of topological operads, rather than cellular decomposition as in~\cite{Ber97} or~\cite{BFV04}. In Section~\ref{secCubes} we introduce the complete graphs operads $\calg_n$ and $\calg^{\mathrm{ext}}_n$ and prove their equivalence. (These are operads in posets, and we denote their geometric realisations by absolute value signs.) This equivalence is then used to relate the little $n$-cubes operad $\calc_n$ to the graphs in the following way: we will construct a commutative diagram (of symmetric sequences of spaces)

\[
  \begin{tikzcd}
    & \Conf_n \ar[r, iso] \ar[dd, iso', "\psi"] & \calc_n \ar[dd, "\varphi"] \\
     & &\\
    & \abs{\calg_n} \ar[r, hook, iso] &\abs{\calg_n^{\mathrm{ext}}}
  \end{tikzcd}
\]

The top map in this diagram is the evident homotopy equivalence relating the configurations of cubes to the configurations of their centres. The objects $\calc_n$, $\abs{\calg_n}$ and $\abs{\calg^{\mathrm{ext}}_n}$ are operads and $\varphi$ is a (lax) morphism of operads, but $\Conf_n$ is not. We will show that the map $\psi: \Conf_n\to\abs{\calg_n}$ is a weak homotopy equivalence between symmetric sequences of spaces, cf. Theorem~\ref{propGisEn} below. Further, the inclusion of $\calg_n$ into $\calg_n^{\mathrm{ext}}$ has been shown to be an equivalence, as well (Theorem~\ref{thmGraphs}), which shows that $\calc_n\to\abs{\calg_n^ {\mathrm{ext}}}$ is an equivalence. In particular, this proves that $\calg_n$ and $\calg_n^{\mathrm{ext}}$ are $E_n$-operads. Earlier attempts to prove this make use of a smaller symmetric sequence $\PP_n\subseteq \calg_n$ of posets, consisting of so called {\em proper\/} graphs (cf.~\cite{Ber97} and~\cite{Kas93}). We comment on the role of these proper graphs in the second part of the section.

In Section~\ref{secBE}, we introduce the Barratt-Eccles operad $\Gamma$ and its filtration by smaller operads $\Gamma_n$. There is a (lax) morphism of operads $\gamma: \Gamma \to \calg$ which maps $\Gamma_n$ into $\calg_n$, and it turns out to be surprisingly easy to show that the latter maps are homotopy equivalences. Having established that $\calg_n$ is an $E_n$-operad, this gives an easy (and new) proof of the conjecture by Smith mentioned above.

In a next section, we consider the Fulton-McPherson operad $\FM_n$. The spaces in this operad carry a stratification by trees, of exactly the same shape as the stratification of the Boardman-Vogt resolution $W\calg_n$ of the complete graph operad $\calg_n$. This results in a map $\FM_n \to W\br{\calg_n}$, and the equivalence $\psi: \Conf_n \to\calg_n$ shows that this map is an equivalence on strata. The fact that $\FM_n$ is equivalent to $\calg_n$ (Theorem~\ref{thmFMisEn} below) then follows easily.

There are two more operads on our list above. The operad $\calm_n$ coming from $n$-fold monoidal categories and the operad $\MS_n$ of McClure and Smith. Following~\cite{BFSV98} we will show that $\calm_n$ is an $E_n$-operad by relating it to the suboperad of $\calc_n$ given by the decomposable (see~\cite{Dun88}) little $n$-cubes. We conclude our paper with a presentation of the McClure-Smith operad in Section~\ref{secMS} and the construction of an equivalence to the complete graphs operad $\calg_n$, showing that it is indeed an $E_n$-operad. The ideas for this last proof are mostly contained in~\cite{MS04}, although we hope to have simplified the presentation somewhat, for example by making optimal use of the formalism of multisimplicial sets.

The paper has two appendices: one explaining some notations concerning the latter formalism, and one in which we explain in more detail the relation of our proofs to the arguments occurring in the literature. 

\subsection*{Acknowledgements}

This paper grew out of several activities on the topic: a seminar at Utrecht, the master's thesis by one of the authors under supervision of the other, and a crash course by the second author at Sheffield. We would like to thank Lennart Meier for useful comments during the seminar, the active audience at Sheffield, and Clemens Berger for helpful correspondence. Further, we would like to thank Vladimir Dotsenko for providing helpful feedback on the preprint of this paper. Finally, we would like to express our thanks to Paolo Salvatore for helping us improve the exposition, and in particular for helping us to correct the formulation of one of the statements in an earlier version of the paper.

\section{Operads}\label{secOperads}
\begin{definition}
  We fix a concrete category $\cale$. In this paper, we will only use the categories of sets, simplicial sets, a convenient category of spaces, posets, categories, or diagrams in one of these categories.
  An {\em operad\/} $P$ in $\cale$ assigns to each finite set $A$ an object $P\br{A}$ in $\cale$ of ``operations taking an A-indexed family of inputs''. Further, we require structure maps
  \begin{itemize}
  \item $1_a\in P\br{\cb{a}}$, the {\em unit}.
  \item For every $\sigma: B\isoto A$ a map $\sigma_\ast: P\br{B}\to P\br{A}$
  \item For every $A, B$, and $a\in A$ a {\em composition\/} map
    \[ \circ_a: P\br{A}\times P\br{B}\to P\br{A\sbr{B/a}} \]
    where $A\sbr{B/a}\coloneqq A\setminus \cb{a}\amalg B$ ($A$ with $a$ replaced by $B$).
  \end{itemize}
  These should satisfy the natural compatibility conditions for symmetries, units, and associativity. For a detailed description, see~\cite{May72}. 
\end{definition}

\begin{remark}
  Since the collection of all finite sets forms a class, we will technically have to work with only the finite sets contained in some convenient universe/level of the von Neumann hierarchy.
  We will disregard this point in the rest of the paper.
\end{remark}

\begin{example}
  \begin{enumerate}[itemindent=*,leftmargin=0pt]
  \item The {\em commutative operad\/} $\Com$ is given by $\Com\br{A}=\cb{\ast}$ for every $A$.
  \item The {\em associative operad\/} $\Ass$ is given by $\Ass\br{A}=\Sigma_A\coloneqq \cb{\text{bijections } A\isoto A}$.
  \end{enumerate}
\end{example}

\begin{remark}
  \begin{enumerate}[itemindent=*,leftmargin=0pt]
  \item Every finite set is isomorphic to a set of the form $n=\cb{0,\dots,n-1}$ for $n\in \N$. With this, it is enough to specify the $P\br{n}$, as is usually done in the literature. The definition above is the ``coordinate free'' version.
  \item Instead of the $\circ_a$ composition, we could equivalently specify composition maps for functions $f:B\to A$ between finite sets:
    \[ P\br{A}\times \prod_{a\in A} P\br{f^ {-1}\br{a}}\to P\br{B} \]
  \end{enumerate}
\end{remark}

\begin{example}[Little $n$-Cubes]
  The operad $\calc_n$ of little $n$-cubes has as space of operations $\calc_n\br{A}$ the space of injections
  \[ \coprod_{a\in A}\br{0,1}^ n\xhookrightarrow{\coprod_{a\in A}c_a} \br{0,1}^ n \]
  such that the component maps $c_a$ are rectilinear (i.e.\ maps of the form $c_a\br{x_1,\dots,x_n}$ $=\br{u_1x_1+v_1,\dots,u_{n}x_n+v_n}=u\odot x+v$ where we are using $\odot$ for the pointwise multiplication of vectors and where all the $u_i$ are positive).
  As for the operad structure, the unit is the identity embedding, the action is by permuting the $c_a$ (i.e.\ for $\sigma:B\isoto A$ and $c\in \calc_n\br{B}$, we set $\sigma_\ast c_{a}=c_{\sigma^ {-1}\br{a}}$), and the composition is defined by composing embeddings
  \begin{align*}
    c^ A\circ_a c^ B: \coprod_{x\in A\sbr{B/a}}\br{0,1}^n&=\coprod_{x\in A\setminus\cb{a}}\br{0,1}^n\amalg \coprod_{b\in B}\br{0,1}^n\\
                                                        &\xhookrightarrow{\id \amalg c^ B} \coprod_{x\in A\setminus\cb{a}}\br{0,1}^n\amalg \br{0,1}^n\\
                                                        &\cong \coprod_{x\in A}\br{0,1}^ n\overset{c^ A}{\hookrightarrow} \br{0,1}^ n
  \end{align*}
  \[
    \adjustbox{scale=0.85}{
    \begin{tikzpicture}[anchor=base, baseline=(X.base),scale=0.7]
      \draw (0,0) rectangle (5,5);
      \draw (1,3) rectangle (4,4);
      \node at (2.5,3.3) {$a_1$};
      \draw (1,1) rectangle (2,2);
      \node (X) at (2.3,2.3) {};
      \node at (1.5,1.3) {$a_2$};
      \draw (2.5, 0.5) rectangle (4.5, 2.5);
      \node at (3.5, 1.3) {$a_3$};
    \end{tikzpicture}\qquad $\circ_{a_3}$ \qquad
    \begin{tikzpicture}[anchor=base, baseline=(X.base), scale=0.7]
      \draw (0,0) rectangle (5,5);
      \draw (0.5,0.5) rectangle (2.5,2.5);
      \node (X) at (2.3,2.3) {};
      \node at (1.5,1.3) {$b_1$};
      \draw (2.5,2.5) rectangle (4.5,4.5);
      \node at (3.5,3.3) {$b_2$};
    \end{tikzpicture}\qquad =\qquad
    \begin{tikzpicture}[anchor=base, baseline=(X.base), scale=0.7]
      \draw (0,0) rectangle (5,5);
      \draw (1,3) rectangle (4,4);
      \node at (2.5,3.3) {$a_1$};
      \draw (1,1) rectangle (2,2);
      \node at (1.5,1.3) {$a_2$};
      \node (X) at (2.3,2.3) {};
      \draw[dashed] (2.5, 0.5) rectangle (4.5, 2.5);
      \draw (2.7, 0.7) rectangle (3.5, 1.5);
      \node at (3.1, 0.9) {$b_1$};
      \draw (3.5, 1.5) rectangle (4.3,2.3);
      \node at (3.9, 1.7) {$b_2$};
    \end{tikzpicture}}
  \]
\end{example}
\vspace{1em}

\begin{definition}\label{defConf}
  The {\em configuration space\/} $\Conf_n\br{A}$ is the space of injections $A\to \br{0,1}^ n$ where we are viewing $A$ as a discrete space --- equivalently, this is the space of configurations of $A$-indexed families of distinct points in $\br{0,1}^ n$. For one such injection $x:A\to \br{0,1}^ n$, we will often write $x_a$ for the image of $a\in A$ under $x$.
\end{definition}

\begin{remark}
  We can define a homotopy equivalence $\gamma: \calc_n\br{A}\to \Conf_n\br{A}$ by sending each cube to its centre. A homotopy inverse is given by replacing each point of a configuration with a sufficiently small cube (e.g.~the cube whose side lengths are half the distance to the closest other point or boundary of the cube $\br{0,1}^n$).
  The configuration spaces do not form an operad. They are, however, what is sometimes called a pre-operad: a contravariant functor from the category of finite, non-empty sets and injections to the category of spaces.
\end{remark}

\section{Equivalences}\label{secEquivalence}
\begin{definition}
  Let $P$ and $Q$ be operads in some category $\cale$. A morphism $\varphi: P\to Q$ of operads is given by a family of morphisms $\varphi_A: P\br{A}\to Q\br{A}$, one for each finite set, which together respect the operad structures of $P$ and $Q$. For example, we want the maps to be compatible with the composition in the sense that the diagram
  \[
    \begin{tikzcd}
      P\br{A}\times P\br{B}\ar[r, "\varphi_A\times \varphi_B"] \ar[d, "\circ_a"'] & Q\br{A}\times Q\br{B}\ar[d,"\circ_a"]\\
      P\br{A\sbr{B/a}} \ar[r,"\varphi_{A\sbr{B/a}}"'] & Q\br{A\sbr{B/a}}
    \end{tikzcd}
  \]
  should commute. For more details, see~\cite{May72}.

  In the case of operads in spaces, we call such a morphism an {\em equivalence}, if each $\varphi_A:P\br{A}\to Q\br{A}$ is a weak homotopy equivalence, and we write $\varphi: P\isoto Q$ in this case. Two operads in spaces are said to be equivalent if they can be related by a zigzag of equivalences
  \[ P\isofrom P_1\isoto P_2\isofrom \dots \isoto Q \]
  In fact, it suffices to consider zigzags of the form $P\isofrom P_1\isoto Q$. In this case, we will write $P\simeq Q$.
\end{definition}

\begin{definition}[$E_n$-operad]
  An operad $P$ in spaces is said to be an {\em $E_n$-operad\/} if it is equivalent to the little $n$-cubes operad and, further, the $\Sigma_n$-action on $P\br{n}$ is free.
\end{definition}

We wish to extend these definitions to operads in other categories.

\begin{definition}
  We call a morphism $P\to Q$ between operads in simplicial sets, posets, or categories an {\em equivalence}, if the induced map $\varphi: \abs{P}\to \abs{Q}$ between the geometric realisations is an equivalence of operads in spaces. Similarly, we will call an operad $P$ in one of these categories an {\em $E_n$-operad}, if $\abs{P}$ is.
\end{definition}

In showing that certain operads are $E_n$-operads, we will often need a weaker version of morphisms between them.

\begin{definition}
  Let $P$ and $Q$ be operads in posets. A {\em lax morphism\/} $\varphi: P\to Q$ is given by a family $\varphi_A:P\br{A}\to Q\br{A}$ of maps of posets, compatible with the symmetries, but preserving composition only in the weak sense that \[ \varphi\br{p\circ_a p^ \prime} \leq \varphi\br{p}\circ_a \varphi\br{p^ \prime} \]
  for any $p\in P\br{A}, p^ \prime\in P\br{B}$, and $a\in A$. We say that a lax morphism $\varphi$ is a {\em lax equivalence}, if each $\varphi_A: P\br{A}\to Q\br{A}$ induces a weak equivalence of classifying spaces $\abs{P\br{A}}\to \abs{Q\br{A}}$.
\end{definition}

\begin{remark}
  Everything we will say about lax morphisms also holds for morphisms that are, instead, lax in the opposite sense that
  \[ \varphi\br{p}\circ_a \varphi\br{p^ \prime} \leq \varphi\br{p\circ_a p^ \prime}\]
  This can be seen by simply replacing $P$ and $Q$ by the operads $P^{\op}$ and $Q^ {\op}$ all of whose posets $P^ {\op}\br{A}$ are simply the opposite of the posets $P\br{A}$ (and analogously for $Q$).
\end{remark}

\begin{lemma}
  If $\varphi: P\to Q$ is a lax equivalence between operads in posets, then $P$ and $Q$ are equivalent operads.
\end{lemma}
\begin{proof}
  Define a new operad $R$ in posets by
  \[ R=\cb{\br{p,q}\in P\br{A}\times Q\br{A}\;\vert\; \varphi\br{p}\leq q} \]
  with the pointwise ordering. Then $R$ has the structure of an operad making the two projections $P\leftarrow R \to Q$ morphisms of operads (this uses that $\varphi$ is a lax morphism). It suffices to see that each of these is an equivalence. To this end, we consider for each set $A$ the diagram of posets
  \[
    \begin{tikzcd}
      P\br{A}\ar[r, shift right,"\iota"'] \ar[rd, "\varphi_A"'] & R\br{A} \ar[l, shift right, "\pi"'] \ar[d, "\lambda"] \\
      & Q\br{A}
    \end{tikzcd}
  \]
  where $\pi$ and $\lambda$ are the projections and $\iota\br{p}=\br{p,\varphi\br{p}}$. Then $\pi$ and $\iota$ are an adjunction of posets and, hence, induce a weak equivalence. Since $\varphi_A=\lambda\circ\iota$ is a weak equivalence by assumption, $\lambda$ is one as well.
\end{proof}

We will need a variation of this to relate operads in different categories.

\begin{definition}\label{defLax}
  Let $P$ be an operad in spaces and $Q$ one in posets. A {\em lax morphism\/} $\varphi: P\to Q$  is a lax morphism of operads in posets, where we view $P$ as an operad in posets by forgetting the topology and giving it the discrete partial order.
\end{definition}

From such a lax morphism, we wish to construct an actual zigzag of morphisms between operads by applying the following procedure to each component map:

Let $\varphi: X\to \mathbb{P}$ be a map (of sets) from a topological space to a poset. We define $X_\mathbb{P}$ to be the topological poset of pairs $\br{x,p}$ with $x\in X, p\in \mathbb{P}$, and $\varphi\br{x}\leq p$, ordered by $\br{x,p}\leq \br{y,q}$ if and only if $x=y$ and $p\leq q$. We write $\abs{X_{\mathbb{P}}}$ for the geometric realisation of this poset. The projections $X\xleftarrow{\mu} X_{\mathbb{P}}\xrightarrow{\pi} \mathbb{P}$ then induce maps of spaces $X\xleftarrow{\mu} \abs{X_{\mathbb{P}}}\xrightarrow{\pi} \abs{\mathbb{P}}$. Further, for each $p\in \PP$, we write $X_{\down p}$ for the subspace of $X$ given by $\cb{x\in X\;\vert\;\varphi\br{x}\leq p}$.

\begin{lemma}\label{lemEquivalence}
  For $\varphi: X\to \mathbb{P}$ and $X_{\mathbb{P}}$ as above
  \begin{enumerate}[label= (\alph*),itemindent=*,leftmargin=0pt]
  \item\label{lemeqa} The map $\mu:\abs{X_{\mathbb{P}}}\to X$ is a weak homotopy equivalence if one of the following holds
    \begin{enumerate}[label= (\roman*),itemindent=*,leftmargin=15pt]
    \item\label{lemeqai} $X$ is a smooth manifold and each $X_{\down p}$ for $p\in \PP$ is open in $X$.
    \item\label{lemeqaii} $X=\colim_{p\in \PP}X_{\down p}$, the poset $\PP$ is well-founded (i.e.~if there are no infinite descending chains) and for each $p$ the map $\colim[q<p]X_{\down q}\to X_{\down p}$ is a cofibration.
    \end{enumerate}
  \item\label{lemeqb} If each subspace $X_{\down p}$ of $X$ is a retract of a CW-complex, then $\abs{X_{\PP}}$ is (a model for) the homotopy colimit of the functor $\PP\to \Top$ sending $p$ to $X_{\down p}$.
  \item\label{lemeqc} If, further, each of the $X_{\down p}$ is weakly contractible, then $\pi: \abs{X_{\mathbb{P}}}\to \abs{\mathbb{P}}$ is a weak homotopy equivalence, as well.
  \end{enumerate}

\end{lemma}
\begin{proof}
  Recall first that if each $X_{\down p}$ is cofibrant in the Kan-Quillen model structure on spaces, then $\abs{X_{\mathbb{P}}}$ is a model (the classical Bousfield-Kan model, see e.g.~\cite{HM22}, Example 10.34) for the homotopy colimit of the diagram $p\mapsto X_{\down p}$ of spaces over $\mathbb{P}$. This shows part~\ref{lemeqb}. Part~\ref{lemeqc} is a direct consequence of this, as the pointwise weak equivalence of diagrams $X_{\down p}\to \ast$ induces a weak equivalence $\hocolim[\mathbb{P}]X_{\down \bullet}\to \hocolim[\mathbb{P}]\ast$, i.e.\ $\abs{X_{\mathbb{P}}}\isoto \abs{\mathbb{P}}$, proving the lemma.

  For part~\ref{lemeqa}\ref{lemeqaii}, we proceed by a similar argument: if $\PP$ is well-founded, then the generalised Reedy model structure on functors $\PP\to\Top$ exists, where we take all morphisms in $\PP$ to be positive and define the degree of $p$ via transfinite recursion as $\deg\br{p}=\sup_{q<p}\br{\deg\br{q}+1}$. As this makes all the matching objects trivial, it is the same as the projective model structure. Further, the condition in point~\ref{lemeqaii} is precisely that for a functor to be cofibrant in this model structure, hence its colimit $\colim[p]X_{\down p}=X$ is a model for its homotopy colimit, but so is $\abs{X_{\PP}}$.

  Finally, for part~\ref{lemeqa}\ref{lemeqai}, we argue as follows: we use the projective model structure on the category $\calp\br{X}$ of simplicial presheaves on $X$. We will think of the objects of this category as simplicial spaces over $X$ and note that we obtain a left Quillen functor $\abs{\bullet}\colon \calp\br{X}\to \Top/X$ given by geometric realisation (since each open subset of $X$ can be given a CW-structure and, so, is cofibrant). This model structure on $\calp\br{X}$ admits a Bousfield localisation $\calp^\prime\br{X}$ by hypercovers where the weak equivalences are stalkwise weak equivalences of simplicial sets (see Jardine~\cite{Jar87}) and $\abs{\bullet}$ factors as a left Quillen functor $\calp^\prime\br{X}\to \Top/X$ because it sends each hypercover to a weak equivalence, see Dugger and Isaksen~\cite{DI04}. Now suppose that each $X_{\down p}$ is open and consider the two objects $Y$ and $Z$ of $\calp\br{X}$ defined by
  \[ Y_n=\coprod_{p_0\leq\cdots\leq p_n}X_{\down p_0}, \qquad Z_n=\coprod_{q_0,\ldots, q_n}X_{\down q_0}\cap\ldots\cap X_{\down q_n} \]
  There is an evident inclusion $Y_\bullet\hookrightarrow Z_\bullet$ and $Z_\bullet$ is a standard hypercover, so $\abs{Z}\to X$ is a weak equivalence by~\cite{DI04}. The map $Y_\bullet\to Z_\bullet$ is also a weak equivalence, as the stalk of $Y_\bullet$ at a point $x$ is the nerve of $\varphi\br{x}/\PP$, hence also a contractible. Thus, we obtain a weak equivalence $\abs{Y_\bullet}\to X$.
\end{proof}

\begin{definition}
  We will call a map $X\to \PP$ from a topological space to a poset {\em sufficiently cofibrant\/} if it fulfils one of the conditions of part~\ref{lemeqa} of the previous lemma.
\end{definition}

\begin{remark}
  If $X$ is the geometric realisation of a simplicial set and each $X_{\down p}$ is the realisation of a simplicial subset, then condition~\ref{lemeqa}\ref{lemeqaii} is automatically satisfied, i.e.~the map $X\to \PP$ is sufficiently cofibrant.
\end{remark}

We can now apply this lemma to lax morphisms as defined in~\ref{defLax}.

\begin{prop}\label{propEquivalence}
  Let $\varphi\colon P\to Q$ be a lax morphism from an operad $P$ in spaces to an operad $Q$ in posets. Then $\varphi$ induces a zigzag
  \[ P\overset{\mu}{\leftarrow} \abs{R}\overset{\pi}{\to}\abs{Q} \]
  of operads in spaces. If $\varphi\colon P\br{A}\to Q\br{A}$ is sufficiently cofibrant for each finite set $A$, then the morphism $\mu$ is an equivalence. If for each finite set $A$ and each $q\in Q\br{A}$ the space $P\br{A}_{\down q}\coloneqq \cb{p\in P\br{A}\;\vert\; \varphi\br{p}\leq q}$ is a weakly contractible retract of a CW-complex, then $\pi\colon \abs{R}\to \abs{Q}$ is a weak equivalence.
\end{prop}
\begin{proof}
  This follows by applying the previous lemma to each of the maps $P\br{A}\to Q\br{A}$ --- the intermediate operad $\abs{R}$ is the realisation of the operad in simplicial spaces given by \[R_n\br{A}=\cb{\br{p,q_1,\dots,q_n}\in P\br{A}\times Q\br{A}^ n\;\vert\; \varphi\br{p}\leq q_1\leq \cdots \leq q_n}\]
\end{proof}

Summarising, we obtain the result

\begin{cor}\label{corEquivalence}
  Let $\varphi:P\to Q$ be a lax morphism from an operad $P$ in spaces to an operad $Q$ in posets. If for each $A$ and each $q\in Q\br{A}$ the space $\cb{p\in P\br{A}\;\vert\; \varphi\br{p}\leq q}$ is weakly contractible and the map $P\br{A}\to Q\br{A}$ is sufficiently cofibrant, then $\varphi$ induces an equivalence $P\simeq \abs{Q}$.
\end{cor}

\section{The Complete Graphs Operad}\label{secCubes}
\begin{definition}
  A directed graph is called {\em complete}, if the underlying undirected graph is. In graph theory, complete directed graphs are sometimes referred to as {\em tournaments}. Let $W$ be a poset whose elements we will call {\em weights}. A graph is said to be {\em weighted\/} by $W$ if to each edge a weight $w\in W$ is assigned. We will consider weighted directed complete graphs $g$.

  If $a\to b$ is a directed edge in $g$, we write $w\br{a,b}$ for the weight and write $a\overset{w}{\to}b$, if $w=w\br{a,b}$. If $g$ and $h$ are graphs on the same set $A$ of vertices, we say that $g\leq h$, if, for any $a,b\in A$, if $a\overset{v}{\to}b$ in $g$, then either $a\overset{w}{\to}b$ in $h$ with $v\leq w$, or $b\overset{w}{\to}a$ in $h$ with $v<w$ strictly.
  In other words, given a graph $h$, one can get a smaller graph by reducing weights, and possibly reversing arrows if the weight has been reduced strictly. In this way, directed weighted complete graphs on a finite set of vertices form a poset.
\end{definition}

\begin{definition}[Complete Graphs Operad]
  The {\em complete graphs operad\/} is the operad, in posets, where $\calg_n\br{A}$ is the poset of {\em cycle-free\/} or {\em acyclic\/} directed weighted complete graphs on the set $A$ of vertices, weighted by the set $\cb{1,\dots,n}$ with the natural order. The condition that the graph is acyclic means that the directions of the graph define a linear order on the set $A$ of vertices.

  The action $\sigma_\ast:\calg_n\br{B}\to \calg_n\br{A}$ for $\sigma: B\isoto A$ is the obvious one. The operad composition is defined as follows: if $g\in \calg_n\br{A}$ and $h\in \calg_n\br{B}$, then, for $a\in A$, the graph $g\circ_a h$ on the set $A\sbr{B/a}$ is obtained by replacing the vertex $a$ in $g$ by the graph $h$ and equipping the resulting graph with the natural weights and directions. Explicitly, in $g\circ_a h$, we have:
  \begin{align*}
    \begin{cases}
      a^ \prime\overset{w}{\to}a^ {\prime\prime}, &a^\prime,a^{\prime\prime}\in A\setminus \cb{a}\text{ and } a^\prime\overset{w}{\to}a^ {\prime\prime}\text{ in }g\\
      b\overset{w}{\to}b^ {\prime}, &b,b^{\prime}\in B \text{ and } b\overset{w}{\to}b^{\prime}\text{ in }h\\
      a^ \prime\overset{w}{\to}b, &a^{\prime}\in A\setminus \cb{a},b\in B\text{ and } a^\prime\overset{w}{\to}a\text{ in }g\\
      b\overset{w}{\to}a^ {\prime}, &a^\prime\in A\setminus \cb{a}, b\in B\text{ and } a\overset{w}{\to}a^ {\prime}\text{ in }g\\
    \end{cases}
  \end{align*}

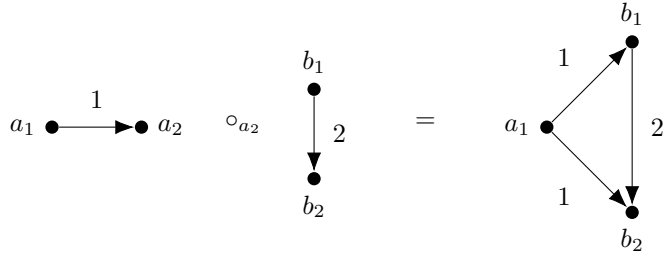
  \[
    \begin{tikzpicture}[anchor=base, baseline]
      \node[circle, fill=black, inner sep=0pt, minimum size=5pt, label=left:{$a_1$}] (0) {};
      \node[circle, fill=black, inner sep=0pt, minimum size=5pt, label=right:{$a_2$}] (1) [right=of 0] {};
      
      \draw[-{Latex[scale=1.5]}] (0) -- (1) node[midway, label=above:{$1$}] {};
    \end{tikzpicture}
    \quad\circ_{a_2}\quad
    \begin{tikzpicture}[anchor=base, baseline={(0,-0.5)}]
      \node[circle, fill=black, inner sep=0pt, minimum size=5pt, label=above:{$b_1$}] (0) {};
      \node[circle, fill=black, inner sep=0pt, minimum size=5pt, label=below:{$b_2$}] (1) [below=of 0] {};
      
      \draw[-{Latex[scale=1.5]}] (0) -- (1) node[midway, label=right:{$2$}] {};
    \end{tikzpicture}\qquad = \qquad
    \begin{tikzpicture}[anchor=base, baseline]
      \node[circle, fill=black, inner sep=0pt, minimum size=5pt, label=left:{$a_1$}] (0) {};
      \node[circle, fill=black, inner sep=0pt, minimum size=5pt, label=above:{$b_1$}] (1) [above right=of 0] {};
      \node[circle, fill=black, inner sep=0pt, minimum size=5pt, label=below:{$b_2$}] (2) [below right=of 0] {};
      
      \draw[-{Latex[scale=1.5]}] (0) -- (1) node[midway, label=above left:{$1$}] {};
      \draw[-{Latex[scale=1.5]}] (0) -- (2) node[midway, label=below left:{$1$}] {};
      \draw[-{Latex[scale=1.5]}] (1) -- (2) node[midway, label=right:{$2$}] {};
    \end{tikzpicture}
  \]
  giving a well-defined element in $\calg_n\br{A\sbr{B/a}}$. This composition is clearly order-preserving, so gives a composition operation $\calg_n\br{B}\times \calg_n\br{A}\to \calg_n\br{A\sbr{B/a}}$.
\end{definition}

\begin{definition}[Variation]
  The {\em extended\/} complete graphs operad $\calg_n^ {\mathrm{ext}}$ is defined in the same way, except that the acyclicity condition is weakened to the condition that there are no cycles of {\em uniform weight}.
\end{definition}

\begin{remark}
  We will write $\abs{\calg_n}$ and $\abs{\calg_n^{\mathrm{ext}}}$ for the operad in topological spaces obtained by taking the classifying spaces of the posets involved (viewed as categories).
\end{remark}

\begin{theorem}\label{thmGraphs}
  The inclusion $\calg_n\hookrightarrow \calg_n^ {\mathrm{ext}}$ induces an equivalence of topological operads $\abs{\calg_n}\hookrightarrow\abs{\calg_n^ {\mathrm{ext}}}$.
\end{theorem}
\begin{proof}
  Let us write $\varphi:\calg_n\hookrightarrow \calg_n^ {\mathrm{ext}}$ for the inclusion. Fix a set $A$. By Quillen's Theorem A~\cite{Qui73}, it suffices to show that, for each $g\in \calg^{\mathrm{ext}}_n\br{A}$, the space $\abs{\varphi/g}$ is contractible (in particular, it is also non-empty). Of course, this is clear if $g\in \calg_n$ (and, hence, for all minimal $g$).

  We proceed by induction on the order on $\calg^ {\mathrm{ext}}_n\br{A}$. Suppose that $g\in \calg_n^ {\mathrm{ext}}\br{A}$ does not belong to $\calg_n\br{A}$, and that we have proved that $\varphi/h$ is contractible for all $h\in \calg_n^ {\mathrm{ext}}\br{A}$ with $h<g$.

  By assumption, $g$ contains cycles, although none of uniform weight. Let us order the cycles as follows: for a cycle $c$ in $g$, write $\mu_i\br{c}$ for the number of edges of weight $i$ in $c$ and $\mu\br{c}\coloneqq \br{\mu_n\br{c},\mu_{n-1}\br{c},\dots,\mu_1\br{c}}$. We write $c\leq d$, if $\mu\br{c}\leq \mu\br{d}$ in the lexicographical order on $\N^ n$. Now, fix a minimal cycle $c$ in $g$ in this ordering, say $c=\br{a_0\overset{w_1}{\to}a_1\overset{w_2}{\to}\dots\overset{w_{k-1}}{\to} a_{k-1}\overset{w_k}{\to}a_k=a_0}$.
  Let us consider the poset $\varphi/g$. If $h\leq g$ belongs to $\calg_n\br{A}$, then $h$ contains no cycles, so $h$ must have reversed at least one edge in $c$ (necessarily of weight $>1$). Let us consider finite subsets $\emptyset\neq U\subseteq \cb{1,\dots,k}$ for which the weight $w_i$ for $i\in U$ on $a_{i-1}\overset{w_i}{\to}a_i$ in $c$ is strictly larger than $1$, and write $g_U$ for the graph obtained from $c$ by reversing the edge $a_{i-1}\overset{w_i}{\to}a_i$ in $c$ and replacing the weight $w_i$ by $w_{i}-1$ for all $i\in U$. Then $g_{U}<g$ and
  \[ \varphi/g=\bigcup_{U}\varphi/g_U \]
  The intersection of $\varphi/g_U$ and $\varphi/g_V$ for two such subsets is again of the form $\varphi/g_W$, as one easily checks. Thus, to show that $\abs{\varphi/g}$ is contractible, it suffices to show that each $g_U$ belongs to $\calg_n^{\mathrm{ext}}\br{A}$, because then the induction hypothesis will give that each $\abs{\varphi/g_U}$ is contractible, hence so is $\abs{\varphi/g}$ (being covered by contractible subspaces with contractible and, in particular, non-empty intersections). We formulate this as the following lemma, which then finishes the proof.
\end{proof}

\begin{lemma}
  Each graph of the form $g_U<g$ belongs to $\calg_n^{\mathrm{ext}}\br{A}$.
\end{lemma}
\begin{proof}
  We must show that $g_U$ does not contain cycles of uniform weight. Suppose to the contrary that $d$ is a cycle of uniform weight in $g_U$. By replacing $U$ by a smaller subset of $\cb{1,\dots,k}$, we may assume that $d$ passes through each of the reversed edges $a_i\to a_{i-1}$ of $c$ for $i\in U$. Pick one such $i\in U$ and let $j$ be the index of the next reversed edge $a_j\to a_{j-1}$ on the cycle $d$. (So if $i<j$ then no number between $i$ and $j$ belongs to $U_j$ and if $i>j$ no number larger than $j$ or smaller than $i$ belongs to $U$, and if $U$ contains just one element, then $i=j$). But then the segments $a_j\to \dots\to a_{i-1}$ of $c$ and $a_{i-1}\to \dots \to a_j$ of $d$ together form a cycle $e$ in the original graph $g$ with $\mu\br{e}<\mu\br{c}$, contradicting the minimality of $c$.
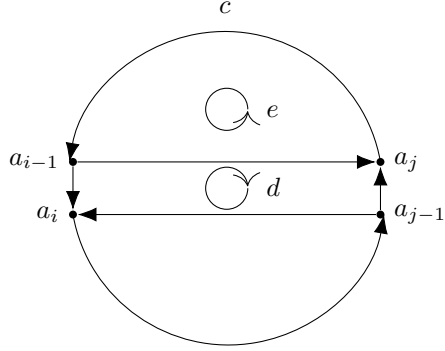
    \[
    \begin{tikzpicture}[scale=0.7]
      \node[circle, fill=black, inner sep=0pt, minimum size=3pt, label=left:{$a_{i-1}$}] (0) at (-2.92,1) {};
      \node[circle, fill=black, inner sep=0pt, minimum size=3pt, label=left:{$a_{i}$}] (1) at (-2.92,0) {};
      \node[circle, fill=black, inner sep=0pt, minimum size=3pt, label=right:{$a_{j}$}] (2) at (2.92,1) {};
      \node[circle, fill=black, inner sep=0pt, minimum size=3pt, label=right:{$a_{j-1}$}] (3) at (2.92,0) {};
      \draw[-{Latex[scale=1.5]}] (0)--(1);
      \draw[-{Latex[scale=1.5]}] (3)--(2);
      \draw[-{Latex[scale=1.5]}] (2) arc (10:170:3cm) node[midway, label=above:{$c$}] {};
      \draw[-{Latex[scale=1.5]}] (1) arc (190:350:3cm);
      \draw[-{Latex[scale=1.5]}] (3)--(1);
      \draw[-{Latex[scale=1.5]}] (0)--(2);
      \draw[-{Computer Modern Rightarrow[scale=2]}] (0.4,0.5) arc (0:-355:0.4cm) node[label=right:{$d$}] {};
      \draw[-{Computer Modern Rightarrow[scale=2]}] (0.4,2) arc (0:355:0.4cm) node[label=right:{$e$}] {};
    \end{tikzpicture}
  \]
\end{proof}

\section{Complete Graphs and Little Cubes}

In this section, we wish to prove that the complete graphs operad is, indeed, an $E_n$-operad. For this, we will first construct a (lax) morphism
\[ \varphi: \calc_n\to \calg_n^ {\mathrm{ext}} \]
from the little $n$-cubes operad to the extended complete graphs operad.

\begin{definition}
  Let $c$ be a point in $\calc_n\br{A}$, given by some embedding of cubes $c_a$ for each $a\in A$. Further, let $i\in \cb{1,\dots,n}$. We say that two cubes $c_a$ and $c_b$ are {\em $i$-separated\/} if they lie on either side of the hyperplane $\cb{x\in \br{0,1}^ n \;\vert\; x_i=t}$ for some $t\in \br{0,1}$. In this case, we will say $c_a$ is {\em $i$-below\/} $c_b$ if $c_a\subseteq \cb{x\in \br{0,1}^ n \;\vert\; x_i<t}$.
  Since $c_a$ and $c_b$ are disjoint, they must be $i$-separated for at least one $i$.
\end{definition}

Now, let $\varphi\br{c}$ be the graph on the set $A$ of vertices whose edges are given as follows: we assign the weight $i$ to the edge between $a$ and $b$ if $i$ is the {\em smallest\/} number such that $c_a$ and $c_b$ are $i$-separated. Further, we direct the edge from $a$ to $b$ if $c_a$ is $i$-below $c_b$, and from $b$ to $a$, otherwise.
This construction defines a directed weighted complete graph which does not contain a cycle of uniform weight: indeed, such a circle would give us numbers $t_0,\dots,t_k\in \br{0,1}$ with $t_0<t_1<\cdots<t_k<t_0$, which is impossible.

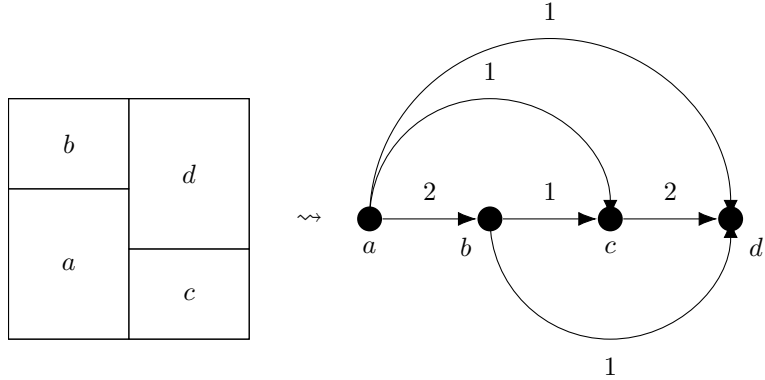
\[
  \begin{tikzpicture}[scale=0.8]
    \draw (-1,-1) rectangle (3,3);
    \draw (-1,-1) rectangle (1,1.5);
    \draw (-1,1.5) rectangle (1,3);
    \draw (1,-1) rectangle (3,0.5);
    \draw (1,0.5) rectangle (3,3);
    \node at (0, 2.25) {$b$};
    \node at (0, 0.25) {$a$};
    \node at (2, 1.75) {$d$};
    \node at (2, -0.25) {$c$};
    \node at (4,1) {$\rightsquigarrow$};
    \node[circle, fill, label=below:$a$] (a2) at (5,1) {};
    \node[circle, fill, label=below left:$b$] (b2) at (7,1) {};
    \node[circle, fill, label=below:$c$] (c2) at (9,1) {};
    \node[circle, fill, label=below right:$d$] (d2) at (11,1) {};
    \draw[-{Latex[scale=1.5]}] (a2) -- (b2) node[midway, label=above:{$2$}] {};
    \draw[-{Latex[scale=1.5]}] (b2) -- (c2) node[midway, label=above:{$1$}] {};
    \draw[-{Latex[scale=1.5]}] (c2) -- (d2) node[midway, label=above:{$2$}] {};
    \draw[-{Latex[scale=1.5]}] (a2) arc (180:0:3) node[midway,label=above:{$1$}] {}; 
    \draw[-{Latex[scale=1.5]}] (a2) arc (180:0:2) node[midway,label=above:{$1$}] {};
    \draw[-{Latex[scale=1.5]}] (b2) arc (-180:0:2) node[midway,label=below:{$1$}] {};
  \end{tikzpicture}
\]

Thus, $\varphi$ gives a well-defined family of maps $\calc_n\br{A}\to \calg_n^ {\mathrm{ext}}\br{A}$.

\begin{lemma}
  The map $\varphi$ is a lax morphism of operads.
\end{lemma}
\begin{proof}
  Let $c\in \calc_n\br{A}, d\in \calc_n\br{B}$, and $a\in A$. Consider $c\circ_a d\in \calc\br{A\sbr{B/a}}$. If $c_a$ and $c_{a^ \prime}$ are $i$-separated, then $d_{b}$ and $c_{a^ \prime}$ are $i$-separated in $c\circ_a d$ for each $b\in B$ as the image of $d_{b}$ is contained in that of $c_a$.
\end{proof}

\begin{remark}
  \begin{enumerate}[itemindent=*,leftmargin=0pt]
  \item Note that strict equality does not necessarily hold, as can be seen by the example
    
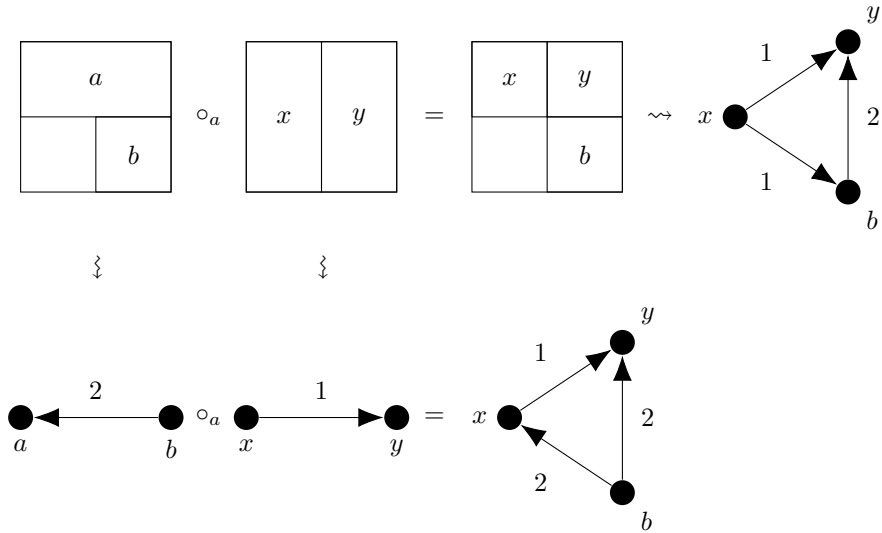
\[
      \begin{tikzpicture}[scale=0.5]
        \draw (0,0) rectangle (4,4);
        \draw (0,2) rectangle (4,4);
        \draw (2,0) rectangle (4,2);
        \node at (2,3) {$a$};
        \node at (3,1) {$b$};
        \node at (5,2) {$\circ_a$};
        \draw (6,0) rectangle (10,4);
        \draw (6,0) rectangle (8,4);
        \draw (8,0) rectangle (10,4);
        \node at (7,2) {$x$};
        \node at (9,2) {$y$};
        \node at (11,2) {$=$};
        \draw (12,0) rectangle (16,4);
        \draw (12,2) rectangle (14,4);
        \draw (14,2) rectangle (16,4);
        \draw (14,0) rectangle (16,2);
        \node at (13,3) {$x$};
        \node at (15,3) {$y$};
        \node at (15,1) {$b$};
        \node at (17,2) {$\rightsquigarrow$};
        \node[fill, circle, label=left:{$x$}] (x) at (19,2) {};
        \node[fill, circle, label=above right:{$y$}] (y) at (22,4) {};
        \node[fill, circle, label=below right:{$b$}] (b) at (22,0) {};
        \draw[-{Latex[scale=2]}] (x)--(y) node[midway, label=above left:{$1$}] {};
        \draw[-{Latex[scale=2]}] (x)--(b) node[midway, label=below left:{$1$}] {};
        \draw[-{Latex[scale=2]}] (b)--(y) node[midway, label=right:{$2$}] {};
        \node at (2,-2) {\rotatebox[origin=c]{270}{$\rightsquigarrow$}};
        \node at (8,-2) {\rotatebox[origin=c]{270}{$\rightsquigarrow$}};
        \node[fill, circle, label=below:{$a$}] (a2) at (0,-6) {};
        \node[fill, circle, label=below:{$b$}] (b2) at (4,-6) {};
        \draw[-{Latex[scale=2]}] (b2)--(a2) node[midway, label=above:{$2$}] {};
        \node at (5,-6) {$\circ_a$};
        \node[fill, circle, label=below:{$x$}] (x2) at (6,-6) {};
        \node[fill, circle, label=below:{$y$}] (y2) at (10,-6) {};
        \draw[-{Latex[scale=2]}] (x2)--(y2) node[midway, label=above:{$1$}] {};
        \node at (11,-6) {$=$};
        \node[fill, circle, label=left:{$x$}] (x3) at (13,-6) {};
        \node[fill, circle, label=above right:{$y$}] (y3) at (16,-4) {};
        \node[fill, circle, label=below right:{$b$}] (b3) at (16,-8) {};
        \draw[-{Latex[scale=2]}] (x3)--(y3) node[midway, label=above left:{$1$}] {};
        \draw[-{Latex[scale=2]}] (b3)--(x3) node[midway, label=below left:{$2$}] {};
        \draw[-{Latex[scale=2]}] (b3)--(y3) node[midway, label=right:{$2$}] {};
      \end{tikzpicture}
    \]
  \item Here, it is indeed necessary to use the extended complete graphs operad rather than the complete graphs operad, as can be seen by the example
    
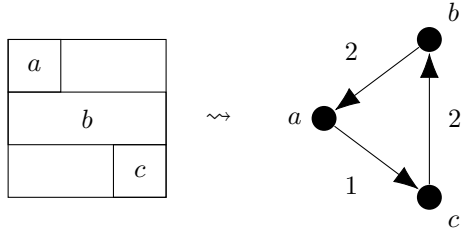
\[
      \begin{tikzpicture}[scale=0.7]
        \draw (0,0) rectangle (3,3);
        \draw (0,2) rectangle (1,3);
        \draw (0,1) rectangle (3,2);
        \draw (2,0) rectangle (3,1);
        \node at (0.5,2.5) {$a$};
        \node at (1.5,1.5) {$b$};
        \node at (2.5,0.5) {$c$};
        \node at (4,1.5) {$\rightsquigarrow$};
        \node[fill, circle, label=left:{$a$}] (a) at (6,1.5) {};
        \node[fill, circle, label=above right:{$b$}] (b) at (8,3) {};
        \node[fill, circle, label=below right:{$c$}] (c) at (8,0) {};
        \draw[-{Latex[scale=2]}] (b)--(a) node[midway, label=above left:{$2$}] {};
        \draw[-{Latex[scale=2]}] (c)--(b) node[midway, label=right:{$2$}] {};
        \draw[-{Latex[scale=2]}] (a)--(c) node[midway, label=below left:{$1$}] {};
      \end{tikzpicture}
    \]
  \end{enumerate}
\end{remark}

\begin{lemma}\label{lemDownCW}
  The sets $\varphi^ {-1}\br{\down g}$ are CW-complexes.
\end{lemma}
\begin{proof}
  For a point $c\in \calc_n\br{A}$ and $a\in A$ given by $c_a=u{\br{a}}\odot x+v\br{a}$ with the pointwise multiplication $\odot$, we will define $w{\br{a}}\coloneqq u\br{a}+v\br{a}$ --- informally, $v\br{a}$ is the lower left and $w\br{a}$ the upper right corner of the cube. Using this notation, $c_a$ is $i$-below $c_b$ if and only if $w\br{a}_i\leq v\br{b}_i$. Thus, if $a\overset{i}{\to} b$ in $g$ and $\varphi\br{c}\leq g$, then either $w\br{a}_j\leq v\br{b}_j$ for $j\leq i$ or $w\br{b}_j\leq v\br{a}_j$ for some $j<i$ strictly.
  Thus, if we identify $\calc_n\br{A}$ with a subspace of $\sbr{0,1}^ {A+A}$ via $c\mapsto \br{v\br{a}, w\br{a}}_{a\in A}$ then
  \[ \varphi^ {-1}\br{\down g}=\bigcap_{a\overset{i}{\to}b\text{ in }g} \br{\bigcup_{j\leq i}\cb{w\br{a}_j\leq v\br{b}_j}\cup \bigcup_{j<i}\cb{w\br{b}_j\leq v\br{a}_j}}\]

  If all the sets of the form $\cb{\dots\leq \dots}$ on the right were to be interpreted in $\sbr{0,1}^ {A+A}$ instead, then this would be a union of simplices intersecting along their faces --- hence, a CW-complex. For those points coming $\calc_n\br{A}$, however, we have the additional requirement that $v\br{a}<w\br{a}$, pointwise. This means that $\varphi^ {-1}\br{\down g}$ is obtained by omitting faces of the form $v\br{a}_i=w\br{a}_i$, but such a space is still a CW-complex.
  
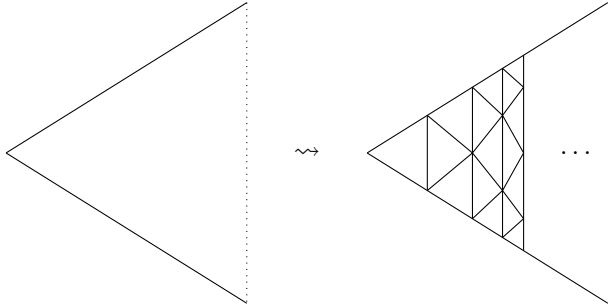
\[
    \begin{tikzpicture}[yscale=0.5, xscale=0.8]
      \draw (0,0) -- (4,4);
      \draw (0,0) -- (4,-4);
      \draw[dotted] (4,4) -- (4,-4);
      \node at (5,0) {$\rightsquigarrow$};
      \draw (6,0) -- (10,4);
      \draw (6,0) -- (10,-4);
      \node at (9.5,0) {$\dots$};
      \draw (7,1) -- (7,-1);
      \draw (7,1) -- (7.75,0);
      \draw (7,-1) -- (7.75,0);
      \draw (7.75,1.75) -- (7.75,-1.75);
      \draw (7.75,1.75) -- (8.25,1);
      \draw (7.75,-1.75) -- (8.25,-1);
      \draw (7.75,0) -- (8.25,1);
      \draw (7.75,0) -- (8.25,-1);
      \draw (8.25,2.25) -- (8.25,-2.25);
      \draw (8.25,2.25) -- (8.6,1.75);
      \draw (8.25,-2.25) -- (8.6,-1.75);
      \draw (8.25,1) -- (8.6,1.75);
      \draw (8.25,-1) -- (8.6,-1.75);
      \draw (8.25,1) -- (8.6,0);
      \draw (8.25,-1) -- (8.6,0);
      \draw (8.6,2.6) -- (8.6,-2.6);
    \end{tikzpicture}
  \]
\end{proof}

\begin{remark}\label{remGraphsConf}
  Analogously to the definition of $\varphi$, we can define a map $\psi: \Conf_n\to \calg_n^ {ext}$ by sending a configuration $x$ on $A$ to the graph on the vertices $A$ with the edge between $a$ and $b$ having the weight $i$, if $i$ is the smallest number such that $x\br{a}_i\neq x\br{b}_i$. Further, we direct it from $a$ to $b$, if $x\br{a}_i<x\br{b}_i$, and from $b$ to $a$, otherwise. One can see that the homotopy equivalence $\gamma:\calc_n\to \Conf_n$ restricts to a homotopy equivalence $\varphi^ {-1}\br{\down g}\to \psi^ {-1}\br{\down g}$ --- the homotopy inverse is defined similarly to that of $\gamma$, but one needs to be a bit more careful about how one chooses the size of the cubes.

  The image of this map, in fact, lands in $\calg_n$: in $\psi\br{x}$, we will have an edge from $a\to b$ if $x\br{a}<x\br{b}$ in the lexicographical order of the coordinates --- this is a linear order and so there cannot be any cycles.
\end{remark}

The following proof is inspired by the argument of Kashiwabara in~\cite{Kas93}, where he shows the equivalence between the filtration stage $\Gamma_n$ of the Barratt-Eccles operad and the configuration spaces $\Conf_n$.

\begin{theorem}\label{propGisEn}
  The map $\psi: \Conf_n\to \calg_n$ induces an equivalence of preoperads $\Conf_n\simeq \abs{\calg_n}$.
\end{theorem}
\begin{proof}
  Let us fix some finite set $A$, an element $a\in A$, and write $B\coloneqq A\setminus\cb{a}$. Further, let $p\coloneqq \abs{B}$. Then there exists an obvious restriction map $\rho: \calg_n\br{A}\to \calg_n\br{B}$.
  Now, for a fixed element $g\in \calg_n\br{B}$, we wish to study the fibre of $\rho$ over it. Note that $g$ defines some linear order $b_1\to\dots\to b_p$ on $B$ --- any extension of $g$ to $A$ then places $a$ either to the very front of this linear order, or just behind some $b_i$ for $i\in \cb{1,\dots,p}$ (we will say that $a$ is in the $0$-th position, if it is placed at the very front, and in the $i$-th position, if it is placed right behind $b_i$). For $i\in \cb{0,\dots,p}$, let us denote by $g_i$ the graph that places $a$ into the $i$-th position and labels every edge into or out of it by $n$.
  \[ g_i:\quad
    \begin{tikzcd}
      b_1\ar[r] \ar[rrrd,"n"'] & \dots \ar[r]& b_i \ar[dr,"n"'] \ar[rr] && b_{i+1}\ar[r] & \dots\ar[r] & b_p\\
      &&&a\ar[ur,"n"'] \ar[urrr, "n"']
    \end{tikzcd}
  \]
  
  Further, for $k\in\cb{1,\dots,n}$ and $i\in\cb{0,\dots,p-1}$, we will write $h_i^+\sbr{k}$ for the graph obtained from $g_{i+1}$ by changing the label of the edge $b_{i+1}\to a$ to $k$, and $h_i^-\sbr{k}$ for the graph obtained from $g_{i}$ by changing the label of the edge $a\to b_{i+1}$ to $k$.
  \[ h^-_i\sbr{k}:\quad
    \begin{tikzcd}
      b_1\ar[r] \ar[rrrd,"n"'] & \dots \ar[r]& b_i \ar[dr,"n"'] \ar[rr] && b_{i+1}\ar[r] & \dots\ar[r] & b_p\\
      &&&a\ar[ur,"k"'] \ar[urrr, "n"']
    \end{tikzcd}
  \]

    \[ h^+_i\sbr{k}:\quad
    \begin{tikzcd}
      b_1\ar[r] \ar[rrrd,"n"'] & \dots \ar[r]& b_{i+1} \ar[dr,"k"'] \ar[rr] && b_{i+2}\ar[r] & \dots\ar[r] & b_p\\
      &&&a\ar[ur,"n"'] \ar[urrr, "n"']
    \end{tikzcd}
  \]

  With these definitions, we note the following properties:
  \begin{enumerate}[itemindent=*,leftmargin=0pt]
  \item $\rho/g=\bigcup_{i}\down g_i$.
  \item $\down g_s\cap \down g_{t+1}\subseteq \down g_t\cap \down g_{t+1}$ for $s\leq t$.
  \item $\down g_t\cap \down g_{t+1}=\down h_t^-\sbr{n-1}\cup \down h_t^{+}\sbr{n-1}$.
  \item $\down h_t^{-}\sbr{1}\cap \down h_t^{+}\sbr{1}=\emptyset$ and $\down h_t^{-}\sbr{k}\cap \down h_t^{+}\sbr{k}=\down h_t^{-}\sbr{k-1}\cup \down h_t^{+}\sbr{k-1}$ for $k>1$.
  \end{enumerate}

  This means that the classifying space of $\down g_t\cup \down g_{t+1}$ is an $\br{n-1}$-sphere and the sequence $\down g_0\subseteq \down g_0\cup \down g_1\subseteq \dots \subseteq \bigcup_i \down g_i$ corresponds to a sequence of attaching spheres. Thus, we have
  \[ \abs{\bigcup_{i=0}^p \down g_i}\simeq \bigvee_{i=1}^p S^{n-1}=\bigvee_{B}S^{n-1} \]
  Also, note that the attachment of spheres depends only on the order, but not on the weights of $g$, where the order corresponds to the order in which the copies of $S^{n-1}$ are attached. Thus, for $g\leq h$, we obtain that the map $\abs{\rho/g}\to \abs{\rho/h}$ is given by some permutation of the copies of the spheres.
  \[
    \begin{tikzcd}
      \abs{\rho/g}\ar[r] \ar[d, iso'] & \abs{\rho/h} \ar[d, iso]\\
      \bigvee_B S^{n-1}\ar[r,"\tau"'] & \bigvee_B S^{n-1}
    \end{tikzcd}
  \]
  By Quillen's theorem B~\cite{Qui73}, this means that the homotopy fibre of $\rho$ is $\bigvee_BS^{n-1}$.

  Let $c\in\Conf_n\br{B}$ be some configuration with $\psi\br{c}\leq g$. Consider the diagram
  \[
    \begin{tikzcd}
      \R^n\setminus c\br{B}\ar[d] \ar[r,"\psi_c"] & \rho/g \ar[d]\\
      \Conf_n\br{A} \ar[d] \ar[r, "\psi_A"] & \calg_n\br{A}\ar[d]\\
      \Conf_n\br{B} \ar[r, "\psi_B"] & \calg_n\br{B}
    \end{tikzcd}
  \]
  The maps $\psi_A$ and $\psi_B$ are the map $\psi$ defined earlier, and the map $\psi_c$ is defined by sending the point $x$ to the image under $\psi_A$ of the configuration given by extending $c$ via $c\br{a}=x$. This makes the diagram commute.

  The column on the left consists of topological spaces, the one on the right of posets and the horizontal maps are maps as discussed in Section~\ref{secEquivalence}. By Lemma~\ref{lemEquivalence} part $(a)$, we obtain a diagram 
  \[
    \begin{tikzcd}
      \R^n\setminus c\br{B}\ar[d] & \abs{\br{\R\setminus\cb{c}}_{\rho/g}} \ar[l, iso'] \ar[r] \ar[d] & \abs{\rho/g} \ar[d]\\
      \Conf_n\br{A} \ar[d] & \abs{\Conf_n\br{A}_{\calg_n\br{A}}} \ar[d] \ar[l, iso'] \ar[r]& \abs{\calg_n\br{A}}\ar[d]\\
      \Conf_n\br{B} & \ar[l, iso'] \ar[r]\abs{\Conf_n\br{B}_{\calg_n\br{B}}} & \abs{\calg_n\br{B}}
    \end{tikzcd}
  \]
  where the horizontal maps on the left are equivalences, as $\R^n\setminus c\br{B}$ and $\Conf_n\br{A}$ are manifolds and the preimages of the downsets are open, as they are defined by strict inequalities. In this diagram, the left hand column is the Fadell-Neuwirth fibration~\cite{FN62}, so the middle column is a fibration up to homotopy. By Quillen's theorem B, the right hand column is as well, as we have just seen. To finish the proof, it now suffices to show that the horizontal maps on the right are weak equivalences, as well. Since this is obviously the case if $A$ is a single point, this follows by induction on the size of $A$ if we can show that the top right horizontal map is a weak equivalence. Thus, the following lemma completes the proof.

  \begin{lemma}
    The map $\psi_c$ induces an equivalence $\abs{\br{\R^n\setminus c\br{B}}_{\rho/g}}\to \abs{\rho/g}$.
  \end{lemma}
  \begin{proof}
    Since we have just shown that the homotopy type of $\rho/g$ depends only on the ordering of $g$, but not the particular $g$, it suffices to prove the claim for one particular $g$. Thus, let us choose $g$ to be the graph $b_1\to\dots\to b_p$ with all edges labelled by $1$ and let $c$ be the configuration with the points $b_i$ arranged from left to right on a line along the first coordinate axis (say, we place $b_i$ at $\br{i,0,\dots,0}$).

    By our analysis of the poset $\rho/g$ at the beginning of the proof of the theorem, it suffices to prove that the spaces $\psi^{-1}_c\br{\downarrow h}$ are contractible CW-complexes, if $h$ is either one of the $g_i$ or one of the $h_i^-\sbr{k}, h_i^+\sbr{k}$.
    \begin{itemize}[itemindent=*,leftmargin=11pt]
    \item For one of the $g_i$, the space $\psi_c^{-1}\br{\down g_i}$ is just $\R^n$ without the line segments of the form $\br{j,0,\dots, 0, t}$ where $j\in\cb{1,\dots,p}$, $t\geq 0$, if $j>i$ and $t\leq 0$, if $j\leq i$. By shrinking those line segments towards $\pm \infty$, it can be seen that this space is homotopic to $\R^n$ and, thus, contractible.
    \item For one of the $h_i^{+}\sbr{k}$, the space $\psi_c^{-1}\br{\down h_i^{+}\sbr{k}}$ is $\R^n$ without the line segments of the form $\br{j,0,\dots, 0, t}$ where $j\in\cb{1,\dots,p}$, $t\geq 0$, if $j>i+1$ and $t\leq 0$, if $j<i+1$, as well as the points of the form $\br{i+1,0,\dots,0,t,\ast,\dots,\ast}$ with $t\leq 0$ where $t$ is the $k$-th entry. The space of points of the latter form can similarly be shrunk towards $\br{i+1,0,\dots,0,-\infty,\ast,\dots,\ast}$. Thus, this space is, equally, contractible.
    \item For one of the $h_i^{-}\sbr{k}$, we get the analogous space to $h_{i}^{-}$ where we, instead, have the points with $t\leq 0$.
    \end{itemize}
  \end{proof}
\end{proof}

\begin{cor}
  The map $\varphi: \calc_n\to \calg_n^{\mathrm{ext}}$ induces an equivalence of operads between $\calc_n$ and $\abs{\calg_n^{\mathrm{ext}}}$.
\end{cor}
\begin{proof}
  This is clear from the diagram
  \[
    \begin{tikzcd}
      \Conf_n\br{A} \ar[r, hook, iso] \ar[d, iso'] & \calc_n\br{A}\ar[d]\\
      \calg_n\br{A} \ar[r, hook, iso'] & \calg_n^{\mathrm{ext}}
    \end{tikzcd}
  \]
  where the map on the left induces an equivalence, as we have just proved, while the map on the bottom is an equivalence by Theorem~\ref{thmGraphs}.
\end{proof}

In the literature, configuration spaces and the little cubes operad have been related to a smaller collection $\PP_n$ of proper graphs; see Appendix~\ref{appB} for precise references. We begin by recalling the definitions and deduce from Theorem~\ref{propGisEn} that the inclusion $\PP_n\hookrightarrow \calg_n$ is an equivalence, as well --- a question left open in the literature.

\begin{definition}\label{defProper}
  For a graph $g\in\calg_n^{\mathrm{ext}}$, we will call $\psi^{-1}\br{g}$ the {\em cell corresponding to $g$}. Further, we will call the subset $\psi^{-1}\br{g}\setminus \bigcup_{h<g}\psi^{-1}\br{h}$ its {\em interior}.

  We will call a graph $g\in \calg_n^ {\mathrm{ext}}$ {\em proper\/} if its corresponding cell has a non-empty interior, i.e.\ if the inclusion
  \[\bigcup_{h<g} \psi^ {-1}\br{h}\subseteq \psi^ {-1}\br{g}\]
  is proper. We will write $\PP_n\br{A}$ for the poset of proper graphs on $A$. The following lemma gives a criterion for a graph to be proper. The proper graphs do not for a suboperad of the complete graphs operad.
\end{definition}

\begin{lemma}
 The poset $\PP_n\br{A}$ consists precisely of those (acyclic) graphs $g$ for which, for any path $v_0\to \dots \to v_k$, the weight of the edge $v_0\to v_k$ is the minimum of the weights of the $v_{i}\to v_{i+1}$.
\end{lemma}
\begin{proof}
 If we have some configuration $x$ of points in the interior of the cell corresponding to $g$, then, in particular, we must have $\psi\br{x}=g$. Now, if $v_0\to\dots\to v_k$ is a path in $g$ and $l$ is the lowest weight on any of the edges, say $v_i\overset{l}{\to}v_{i+1}$, then, by definition of $\psi$, we must have $x\br{v_p}_j=x\br{v_{p+1}}_j$ for all $j<l$ and $x\br{v_p}_l\leq x\br{v_{p+1}}_l$ for all $p\in \cb{0,\dots,k-1}$ with at least one inequality holding strictly. In particular, we have $x\br{v_0}_j=x\br{v_{k}}_j$ for $j<l$ and $x\br{v_0}_l<x\br{v_k}_l$. Thus, the edge $v_0\to v_k$ has weight $l$ in $\psi\br{x}=g$ and the condition is necessary.

  To see that it is sufficient, let $g$ be a graph satisfying it. Without loss of generality, we assume $A=\cb{0,\dots,m}$ and inductively define a configuration on it. We set $x\br{0}=0$ and, if $x\br{i}$ has already been defined and the edge $i\to i+1$ has weight $l$ in $g$, then we set $x\br{i+1}\coloneqq x\br{i}+\br{0,\dots,1,\dots,0}$ with the $1$ in the $l$-th position. It is clear that $\psi\br{x}$ agrees with $g$ on edges of the form $i\to i+1$ and, by the property, it also does on all others.
\end{proof}

\begin{theorem}
  The map $\psi: \Conf_n\to \PP_n$ induces an equivalence $\Conf_n\simeq \abs{\PP_n}$.\label{thmCubesProperGraphs}
\end{theorem}
\begin{proof}
  We will apply Corollary~\ref{corEquivalence} to the map $\psi$. Again, $\Conf_n$ is a manifold and the preimages of the downsets are open. Thus, it remains to be shown that the $\psi^ {-1}\br{\down g}$ are also contractible for all proper $g$. The CW-part is analogous to the argument in Lemma~\ref{lemDownCW}, only that all inequalities are strict, i.e.~we omit all faces. For contractibility, let $g$ be a proper graph and $x\in \Conf_n\br{A}$ be a point with $\psi\br{x}=g$.

  Since $\psi\br{x}=g$, we have that, for every edge $a\overset{i}{\to}b$ in $g$, $x\br{a}_j=x\br{b}_j$ for $j<i$ and $x\br{a}_i<x\br{b}_i$. We set
  \[ Z_k\coloneqq \cb{y\in \psi^ {-1}\br{\down g}\;\vert\; y\br{a}_i=x\br{a}_i \text{ for all }a\in A\text{ and }i>k} \]
  In particular, we have $Z_n=\psi^ {-1}\br{\downarrow g}$ and $Z_0=\cb{x}$. To show that $\psi^ {-1}\br{\down g}$ is contractible, it suffices to see that we can retract $Z_k$ to $Z_{k-1}$.

  We define a retraction as follows:
  for $t\in \sbr{0,1}$, we map a point $y\in Z_k$ to the point $y^{\br{t}}$ given by
  \[
    y^ {\br{t}}\br{a}_i\coloneqq
    \begin{cases}
      y\br{a}_i, & i\neq k\\
      \br{1-t}y\br{a}_k+tx\br{a}_k, &i=k
    \end{cases}
  \]

  It remains to be seen that this is, indeed, a well-defined retraction. In other words, we need to see that $y^{\br{t}}$ is an element of $\psi^ {-1}\br{\down g}$ for all  $t$. Suppose we have some edge $a\overset{i}{\to}b$ in $g$. If $i>k$, then $y^ {\br{t}}\br{a}_i=x\br{a}_i$ and, since $g=\psi\br{x}$, also $y^ {\br{t}}\br{a}_i<y^ {\br{t}}\br{b}_i$. If $i<k$, then $y^ {\br{t}}\br{a}_i=y\br{a}_i$ and $y\in \psi^ {-1}\br{\down g}$. Lastly, we need to check the case $i=k$. If $y\br{a}_j\neq y\br{b}_j$ for some $j<i$, then also $y^ {\br{t}}\br{a}_j\neq y^ {\br{t}}\br{b}_j$. Otherwise, we must have that $y\br{a}_k<y\br{b}_k$, since $y\in\psi^ {-1}\br{\down g}$, but $x$ also satisfies $x\br{a}_k\neq x\br{b}_k$ and therefore we also have $y^ {\br{t}}\br{a}_k\neq y^ {\br{t}}\br{b}_k$.

  Thus, we have $y^{\br{t}}\in \psi^{-1}\br{\downarrow g}$.
\end{proof}

\begin{remark}
  In~\cite{Ber97}, it has been claimed that the inclusion $\PP\br{A}/\Sigma_A\leftrightarrows \calg\br{A}/\Sigma_A$ has a right adjoint, which would make this into a direct proof of the equivalence between $\calc_n$ and $\calg_n$. Unfortunately, such an adjunction cannot exist. If we do not quotient by $\Sigma_A$, the following gives a counterexample:
  
    
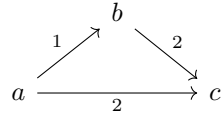
\[
      \begin{tikzcd}
        & b \ar[rd, "2"]& \\
        a\ar[ur, "1"] \ar[rr, "2"'] & & c
      \end{tikzcd}
    \]
    then the poset $\PP_2/g$ is given by
    
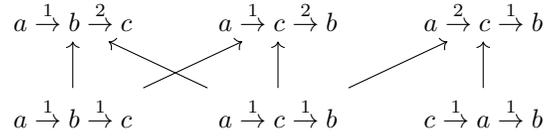
\[
      \begin{tikzcd}
        a\overset{1}{\to}b\overset{2}{\to}c & a\overset{1}{\to}c\overset{2}{\to}b & a\overset{2}{\to}c\overset{1}{\to}b\\
        a\overset{1}{\to}b\overset{1}{\to}c \ar[u]\ar[ur] & a\overset{1}{\to}c\overset{1}{\to}b\ar[u]\ar[ul] \ar[ur]& c\overset{1}{\to}a\overset{1}{\to}b\ar[u]
      \end{tikzcd}
    \]
    and, hence, not contractible (it is homotopy equivalent to $S^1$). So $\PP\br{A}\hookrightarrow \calg\br{A}$ cannot have a right adjoint. The same conclusion then follows for $\PP\br{A}/\Sigma_A\leftrightarrows \calg\br{A}/\Sigma_A$. (Indeed, for a general category $\cale$ with a free left action by a group $H$, the quotient category $\mathrm{Ar}\br{\cale}/H\rightrightarrows\mathrm{Ob}\br{\cale}/H$ is equivalent to the Grothendieck construction $\int_H\cale$. Its slice $\br{\int_H\cale}/x$ over an object $x$ has the slice $\cale/x$ as a deformation retract, so it has the same homotopy type.)

    A left adjoint cannot exist, either. For this, consider the graph
    
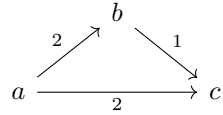
\[
      \begin{tikzcd}
        & b \ar[rd, "1"]& \\
        a\ar[ur, "2"] \ar[rr, "2"'] & & c
      \end{tikzcd}
    \]
    then the poset $g/\PP_2$ consists of only to incomparable elements, $a\overset{2}{\to}b\overset{2}{\to}c$ and $a\overset{2}{\to}c\overset{2}{\to}b$. Therefore, it is not connected and, in particular, not contractible. 
\end{remark}

\begin{cor}
  The inclusion $\PP_n\to \calg_n$ induces an equivalence $\abs{\PP_n}\hookrightarrow \abs{\calg_n}$ of preoperads in spaces.
\end{cor}
\begin{proof}
  This follows from~\ref{thmCubesProperGraphs} and~\ref{propGisEn}.
\end{proof}

\section{The Barratt-Eccles Operad}\label{secBE}

\begin{definition}
  For a finite set $A$, we let $\Lin{A}$ denote the category of {\em linear orders on $A$}. More precisely, we define it to be the category whose objects consist of pairs $\br{L,l}$ with $L$ a linearly ordered set and $l: A\to L$ a bijection. Further, we define a morphism $\br{L,l}\to \br{M,m}$ to consist of a function $f: L\to M$ (not necessarily respecting the orders on $L$ and $M$), such that the triangle
  
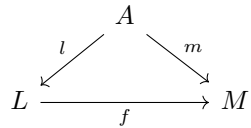
\[
    \begin{tikzcd}
      & A \ar[ld, "l"'] \ar[rd, "m"] & \\
      L \ar[rr, "f"'] & & M
    \end{tikzcd}
  \]
  commutes. Note that any such morphism is necessarily a bijection. Moreover, between any two objects, there is exactly one morphism.

  Given a bijection $\sigma: B \isoto A$, we obtain a functor $\Lin{B}\to \Lin{A}$ given by sending an object $\br{L,l}$ to the object $\br{L, l\circ \sigma^{-1}}$. Further, we can define a composition for these categories via
  \begin{align*}
    \circ_a: &\Lin{A}\times \Lin{B}\to \Lin{A\sbr{B/a}}\\
    &\br{\br{L,l}, \br{M,m}} \mapsto \br{L\sbr{M/l\br{a}}, l\sbr{m/a}}
  \end{align*}
  where $L\sbr{M/l\br{a}}$ is the linear order whose underlying set is $L\sbr{M/l\br{a}}$ and whose order is defined by
  \[
    x<y\colonLeftrightarrow
    \begin{cases}
      x<_L y, & x,y\in L\\
      x<_M y, & x,y\in M\\
      x<_L l\br{a}, & x\in L, y\in M\\
      l\br{a} <_L y, & x\in M, y\in L
    \end{cases}
  \]
  and $l\sbr{m/a}: A\sbr{B/a}\to L\sbr{M/l\br{a}}$ is defined via $l\sbr{m/a}\br{x}=l\br{x}$ for $x\notin B$ and $l\sbr{m/a}\br{y}=m\br{y}$ for $y\in B$. This gives $\mathsf{Lin}$ the structure of an operad.

  Taking the nerves of these categories, for all finite sets $A$, gives the {\em Barratt-Eccles\/} operad $\Gamma$ in $\sSet$. The spaces $\abs{\Gamma\br{A}}$ have a free $\Sigma_A$ action and are contractible --- in other words, $\abs{\Gamma}$ defines an $E_\infty$-operad.
\end{definition}

\begin{remark}
  The classical definition of the Barratt-Eccles operad is slightly different, using the categories whose objects are the permutations on $A$, instead of $\Lin{A}$. If we choose some linear order on $A$, it is clear that these two categories are equivalent.
\end{remark}

\begin{definition}\label{defWeightBE}
  For a $k$-simplex $x$ of $\Gamma\br{A}$, i.e.\ a diagram of the form
  \[
    \begin{tikzcd}
      & & A \ar[lld, "l_0"'] \ar[ld, "l_1"] \ar[rd, "l_k"] & \\
      L_0\ar[r, "f_1"'] & L_1\ar[r, "f_2"'] & \dots \ar[r, "f_k"'] & L_k
    \end{tikzcd}
  \]
  and two elements $a,b\in A$, we define the {\em weight\/} $w_x\br{a,b}$ as $1$ plus the number of $f_i$ that switch $a$ and $b$, or, more precisely, it is $1$ plus the number of indices $i$ such that the order of $l_{i-1}\br{a}$ and $l_{i-1}\br{b}$ is different from that of $f_i\br{l_{i-1}\br{a}}=l_i\br{a}$ and $l_i\br{b}$. For a face map $d_i$ or a degeneracy map $s_j$, we have $w_{d_{i}x}\br{a,b}\leq w_x\br{a,b}$ and $w_{s_{j}x}\br{a,b}=w_x\br{a,b}$, so we can define a simplicial subset
  \[ \Gamma_n\br{A}\coloneqq \cb{x\in \Gamma\br{A}\;\vert\; w_x\br{a,b}\leq n\;\forall a,b\in A} \]

  These spaces give a suboperad $\Gamma_n$ of $\Gamma$. We will return to the exact properties of these suboperads in the next subsection.
\end{definition}
\begin{example}
  For the simplex $x$ given by
  \[
    \begin{tikzcd}
      \dots < a <\cdots <b<\dots\ar[d]\\
      \dots < a <\cdots <b<\dots\ar[d]\\
      \dots < b <\cdots <a<\dots\ar[d]\\
      \dots < a <\cdots <b<\dots
    \end{tikzcd}
  \]
  we have $w_x\br{a,b}=3$.
\end{example}

\subsection{Relation to Complete Graphs}

In this section, we will prove the equivalence between the filtrations of the complete graphs and the Barratt-Eccles operad by constructing a map between the two.

\begin{definition}
  For an $m$-simplex $x\in \Gamma\br{A}$ and $a\neq b\in A$, recall the definition of the weights $w_x\br{a,b}$ from Definition~\ref{defWeightBE}. Assume $x$ is of the form
    
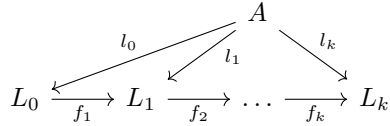
\[
    \begin{tikzcd}
      & & A \ar[lld, "l_0"'] \ar[ld, near start, "l_1"] \ar[rd, "l_k"] & \\
      L_0\ar[r, "f_1"'] & L_1\ar[r, "f_2"'] & \dots \ar[r, "f_k"'] & L_k
    \end{tikzcd}
  \]
  then we will associate to it the graph $\gamma\br{x}$ defined by an edge $a\overset{i}{\to}b$ for any two $a\neq b\in A$, if $w_x\br{a,b}=i$ and $l_0\br{a}<l_0\br{b}$. In other words, the underlying linear order of $\gamma\br{x}$ is $L_0$ and the edges are weighted by the weight in $x$ between their endpoints.

  This defines a function
  \[ \gamma: \Gamma\br{A}_k\to \calg\br{A} \]
  Further, by definition of the suboperads $\Gamma_n$, it is clear that this restricts to a map
  \[ \gamma: \Gamma_n\br{A}\to \calg_n\br{A} \]
  for any $n$.
\end{definition}

\begin{lemma}
  The map $\gamma$ gives a morphism of operads $\abs{\Gamma_n}\to \calg_n$ as in Definition~\ref{defLax} --- in fact, we have equalities instead of inequalities.
\end{lemma}
\begin{proof}
  It suffices to prove that $\gamma$ respects the operad structure as a map from $\Gamma_n$ (as an operad in simplicial sets), as the geometric realisation is functorial.
  For $A,B$ and $a\in A$, we need to consider the functors
  \[ \circ_a: \Lin{A}\times \Lin{B}\to \Lin{A\sbr{B/a}} \]

  If $f$ is a morphism in $\Lin{A}$, and $g$ is a morphism in $\Lin{B}$, then we will write $f\circ_a g$ for the corresponding morphism given by applying the functor, i.e.
  \[
    \adjustbox{scale=0.85}{%
    \begin{tikzcd}
      & A \ar[ld, "l"'] \ar[rd, "l^\prime"] & \\
      L \ar[rr, "f"'] & & L^\prime
    \end{tikzcd}\qquad
    \begin{tikzcd}
      & B \ar[ld, "m"'] \ar[rd, "m^\prime"] & \\
      M \ar[rr, "g"'] & & M^\prime
    \end{tikzcd}\qquad
    \begin{tikzcd}
      & A\sbr{B/a}\ar[ld, "l\sbr{m/a}"'] \ar[rd, "l^\prime\sbr{m^\prime/a}"] & \\
      L\sbr{M/a} \ar[rr, "f\circ_a g"'] & & L^\prime\sbr{M^\prime/a}
    \end{tikzcd}}
  \]

  For $x,y\in A\setminus\cb{a}$ and $z,w\in B$, note the following:
  \begin{enumerate}
  \item $f$ switches $x$ and $y$, if and only if $f\circ_{a}g$ does.
  \item $g$ switches $z$ and $w$, if and only if $f\circ_{a}g$ does.
  \item $f$ switches $a$ and $x$, if and only if $f\circ_a g$ switches $x$ and $z$ (or $w$).
  \end{enumerate}
  Applying this observation to the strings of morphisms of two $k$-simplices, $x\in \Gamma\br{A}_k$ and $y\in \Gamma\br{B}_k$, we see that $\gamma\br{x\circ_a y}=\gamma\br{x}\circ_a \gamma\br{y}$.
\end{proof}

\begin{definition}
  For a graph $g\in \calg_n\br{A}$, we will write $\Gamma_g\subseteq \Gamma_n\br{A}$ for the simplicial set
  \[ \Gamma_g\coloneqq \cb{x\in \Gamma\br{A}\;\vert\; \gamma\br{x}\leq g} \]
\end{definition}

\begin{lemma}
  For any $g\in \calg_n\br{A}$, the space $\abs{\Gamma_g}$ is contractible (and a CW-complex).\label{lemBEDownContractible}
\end{lemma}
\begin{proof}
  Let $L_g$ be the set $A$ equipped with the linear order given by $g$. Then $\br{L_g,\id:A\to L_g}$ (which we, by abuse of notation, will also denote by $L_g$) defines a vertex of $\Gamma_g$. We claim that $\Gamma_g$ is a retract of the simplicial set $L_g/\Gamma_g$ whose $k$-simplices are given by
  \[\br{L_g/\Gamma_g}_k\coloneqq \cb{ x=\br{L_0\to\dots\to L_{k+1}}\in \br{\Gamma_{g}}_{k+1}\;\vert\; L_0=L_g }\]
  Note that this simplicial set is contractible (there is an obvious contraction to the vertex $L_g\overset{\id}{\to}L_g$).

  To prove the claim, we wish to define a section $\sigma$ of the map $\pi: L_g/\Gamma_g\to \Gamma_g$ sending a $k$-simplex $x$ (i.e.\ a certain $k+1$-simplex of $\Gamma_g$) to the simplex given by forgetting the first vertex (viewing $\br{L_g/\Gamma_g}_k$ as a subset of $\br{\Gamma_g}_{k+1}$, this is simply the face map $d_0$ of $\Gamma_g$ (but not the face map of the simplicial set $L_g/\Gamma_g$)).

  To do so, suppose we have a $k$-simplex $x\in \Gamma_g$ of the form
   \[
    \begin{tikzcd}
      A \ar[d, "l_0"'] \ar[rd, "l_1"] \ar[rrrd, "l_k"] & \\
      L_0\ar[r, "f_1"'] & L_1\ar[r, "f_2"'] & \dots \ar[r, "f_k"'] & L_k
    \end{tikzcd}
  \]
  Then we define a $k+1$-simplex $\sigma\br{x}\in \Gamma_g$ by
  \[
    \begin{tikzcd}
      & A \ar[ld, "\id"'] \ar[d, "l_0"'] \ar[rd, "l_1"] \ar[rrrd, "l_k"] & \\
       L_g\ar[r, "l_0"'] & L_0\ar[r, "f_1"'] & L_1\ar[r, "f_2"'] & \dots \ar[r, "f_k"'] & L_k
    \end{tikzcd}
  \]
  If this indeed defines a $k$-simplex of $L_g/\Gamma_g$, then it is clear that $\sigma$ gives a section of $\pi$, so it remains to check just that. To this end, let $a,b\in A$, let $x$ be a $k$-simplex of $\Gamma_g$ (i.e.\ $\gamma\br{x}\leq g$) and suppose $a\overset{i}{\to} b$ in $g$ (in particular, we have $a<b$ in $L_g$).
  If $l_0\br{a}<l_0\br{b}$, as well, then $a\overset{j}{\to} b$ in $\gamma\br{x}$ for some $j\leq i$, and the same holds in $\gamma\br{\sigma\br{x}}$. Otherwise, if $l_0\br{b}<l_0\br{a}$, then we must have $b\overset{j}{\to}a$ in $\gamma\br{x}$ for some $j$ strictly less than $i$, by definition of $\gamma$ and the ordering on $\calg\br{A}$. This means that, for $\gamma\br{\sigma\br{x}}$, we will have $a\overset{j+1}{\to}b$, and $j+1\leq i$. Thus, $\gamma\br{\sigma\br{x}}\leq g$, as required.
\end{proof}

\begin{theorem}\label{thmBEisEn}
  The operad $\abs{\Gamma_n}$ is equivalent to $\calg_n$.
\end{theorem}
\begin{proof}
  This follows by Corollary~\ref{corEquivalence} using the previous two lemmas, as $\abs{\Gamma_n}$ is the geometric realisation of a simplicial set and each $\abs{\Gamma_g}$ is the geometric realisation of a simplicial subset.
\end{proof}

\section{The Fulton-MacPherson Compactification}

We begin with a quick review of the construction of this compactification, and refer to~\cite{Kon99} or~\cite{GJ94} for details.

\begin{definition}\label{defFM}  Recall the definition of the configuration space from~\ref{defConf}. The group $\R_{>0}\ltimes \R^n$ acts on $\Conf_n\br{A}$ via scaling and translating. As this action is free, we obtain a bundle $\Conf_n\br{A}\to \Conf_n\br{A}/\R_{>0}\ltimes\R^n$. This bundle has a section given by scaling an equivalence class of configurations in $\Conf_n\br{A}/\R_{>0}\ltimes \R^n$ so that the maximum distance between two points is equal to $1$ and translating it such that the barycentre is the origin. Thus, the bundle is trivial. Writing $\NConf_n\br{A}$ for the image of this section (we will refer to its image as {\em normalised\/} configurations), we obtain a diagram
  \[
    \begin{tikzcd}
      \Conf_n\br{A}\ar[r, two heads] & \Conf_n\br{A}/\R_{>0}\ltimes \R^n\\
      \NConf_n\br{A} \ar[u, hook] \ar[ur,iso']
    \end{tikzcd}
  \]
  There exists an embedding of $\NConf_n\br{A}$ into the space $\prod_{a,b,c\in A, a\neq b\neq c\neq a} \R_{>0}\times \prod_{a,b\in A, a\neq b}S^{n-1}$ given by sending a configuration $x$ to the point whose coordinate in $\R_{>0}$ for $a,b,c\in A$ is $\abs{\frac{x\br{a}-x\br{c}}{x\br{b}-x\br{c}}}$ and whose coordinate in $S^{n-1}$ for $a,b\in A$ is $\frac{x\br{a}-x\br{b}}{\abs{x\br{a}-x\br{b}}}$.
  We define $\FM_n\br{A}$ to be the compactification of the image of $\NConf_n\br{A}$ in $\prod_{a,b,c\in A, a\neq b\neq c\neq a}\sbr{0,\infty}\times \prod_{a,b\in A, a\neq b}S^{n-1}$.
\end{definition}

\begin{remark}\label{remStratification}
  The space $\FM_n\br{A}$ has a stratification indexed by rooted trees whose set of leaves is $A$ and which do not have any vertices with only one input edge.
  For two such trees, we will write $S\rightarrowtail T$, if $S$ can be obtained from $T$ by contracting some number of inner edges. 

  
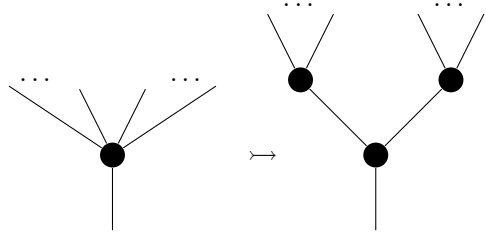
\[
    \begin{tikzpicture}[scale=0.5]
      \node[fill, circle] (A) at (7,0) {};
      \node[fill, circle] (B) at (5,2) {};
      \node[fill, circle] (C) at (9,2) {};
      \node (x) at (8,4) {};
      \node (y) at (9,4) {$\dots$};
      \node (z) at (10,4) {};
      \node (v) at (4,4) {};
      \node at (5,4) {$\dots$};
      \node (w) at (6,4) {};
      \draw (A) -- (7,-2);
      \draw (A) -- (B);
      \draw (C) -- (x);
      \draw (C) -- (z);
      \draw (A) -- (C);
      \draw (B) -- (v);
      \draw (B) -- (w);
      \node at (4,0) {$\rightarrowtail$};
      \node[fill, circle] (A2) at (0,0) {};
      \node (x2) at (1,2) {};
      \node (y2) at (2,2) {$\dots$};
      \node (z2) at (3,2) {};
      \node (v2) at (-3,2) {};
      \node at (-2,2) {$\dots$};
      \node (w2) at (-1,2) {};
      \draw (A2) -- (0,-2);
      \draw (A2) -- (x2);
      \draw (A2) -- (z2);
      \draw (A2) -- (v2);
      \draw (A2) -- (w2);
    \end{tikzpicture}
  \]
  
  The stratification has the following properties:
  \begin{enumerate}
  \item $\FM_n\br{A}$ is partitioned into subspaces $\FM_n\br{T}$ indexed by such trees.
  \item $\FM_n\br{T}\subseteq \overline{\FM_n\br{S}}$, if $S\rightarrowtail T$.
  \item If $C_A$ denotes the corolla whose set of leaves is $A$, then $\FM_n\br{A}\supseteq \NConf_n\br{A}=\FM_n\br{C_A}$.
  \item $\FM_n\br{A}$ is a manifold with corners, its interior is given by $\FM_n\br{C_A}$, and the boundary is the union of the $\FM_n\br{T}$ taken over trees $T$ with at least one inner edge.
  \item For each tree $T$, the stratum $\FM_n\br{T}$ is canonically homeomorphic to a product of normalised configuration spaces
    \[ \FM_n\br{T}\cong \prod_{v\in T}\NConf_n\br{\mathrm{in}\br{v}} \]
    where the product is taken over the vertices $v$ of $T$ and $\mathrm{in}\br{v}$ denotes the set of input edges of $v$.
  \end{enumerate}

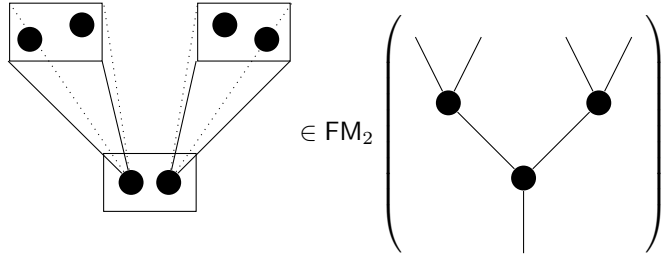
  \[
    \begin{tikzpicture}[scale=0.5, anchor=base, baseline=5]
      \draw (0,0,0) rectangle (4,0,4);
      \node[fill,circle] (A) at (1.5,0,2) {};
      \node[fill,circle] (B) at (2.5,0,2) {};
      \draw (-2.5,4,0) rectangle (1.5, 4, 4);
      \draw (2.5,4,0) rectangle (6.5, 4, 4);
      \node[fill, circle] at (-1,4,2.5) {};
      \node[fill, circle] at (0,4,1.5) {};
      \node[fill, circle] at (3.8,4,1.5) {};
      \node[fill, circle] at (5.3,4,2.5) {};
      \draw (A) -- (-1,4,4);
      \draw (A) -- (1.5,4,4);
      \draw[dotted] (A) -- (-2.5,4,0);
      \draw[dotted] (A) -- (0,4,0);
      \draw (B) -- (4,4,4);
      \draw (B) -- (6.5,4,4);
      \draw[dotted] (B) -- (2.5,4,0);
      \draw[dotted] (B) -- (5,4,0);      
    \end{tikzpicture}\in
    \FM_2\br{
      \begin{tikzpicture}[scale=0.5, anchor=base, baseline]
        \node[fill, circle] (A) at (7,-1) {};
        \node[fill, circle] (B) at (5,1) {};
        \node[fill, circle] (C) at (9,1) {};
        \node (x) at (8,3) {};
        \node (z) at (10,3) {};
        \node (v) at (4,3) {};
        \node (w) at (6,3) {};
        \draw (A) -- (7,-3);
        \draw (A) -- (B);
        \draw (C) -- (x);
        \draw (C) -- (z);
        \draw (A) -- (C);
        \draw (B) -- (v);
        \draw (B) -- (w);
      \end{tikzpicture}
    }
  \]
\end{remark}

\begin{example}
  For a set $A$ with $\abs{A}\leq 2$, the space $\NConf_n\br{A}$ is already compact --- in other words, we have $\FM_n\br{A}=\NConf_n\br{A}$ and the stratification consists only of the single tree $C_A$. 
\end{example}

\begin{definition}
  The spaces $\FM_n\br{A}$ have the structure of an operad called the {\em Fulton-MacPherson operad}. The composition operation is given as follows:

  For finite sets $A$ and $B$, $a\in A$, and trees $S$ and $T$ whose sets of leaves are $A$ and $B$, respectively, as well as elements \[\br{y_v}_{v\in S}\in \prod_{v\in S}\NConf_n\br{\mathrm{in}\br{v}}=\FM_n\br{S}\] and \[\br{x_v}_{v\in T}\in \prod_{v\in T}\NConf_n\br{\mathrm{in}\br{v}}=\FM_n\br{T}\] we can construct a new tree $S\circ_a T$ by grafting $T$ onto the leaf $a$ of $S$. The set of leaves of $S\circ_a T$ is $A\sbr{B/a}$, its set of vertices is the union of the sets of vertices of $S$ and $T$, and the input edges of some vertex are the same as the input edges of the corresponding vertex of $S$ or $T$. Thus, the family $\br{z_v}_{v\in S\circ_a T}$ with $z_v=x_v$ for $v\in T$ and $z_v=y_v$ for $v\in S$ defines a point of \[\prod_{v\in S\circ_a T}\NConf_n\br{\mathrm{in}\br{v}}=\FM_n\br{S\circ_a T}\subseteq \FM_n\br{A\sbr{B/a}}\]
    \[
    \begin{tikzpicture}[scale=0.5]
        \node[fill, circle] (A) at (0,0) {};
        \node[label=above:$a$] (B) at (-2,2) {};
        \node (x) at (0,2) {};
        \node (y) at (1,2) {$\dots$};
        \node (z) at (2,2) {};
        \draw (A) -- (0,-2);
        \draw (A) -- (B);
        \draw (A) -- (x);
        \draw (A) -- (z);
        \node at (3,0) {$\circ_a$};
        \node[fill, circle] (A2) at (6,0) {};
        \node (x2) at (4,2) {};
        \node (y2) at (6,2) {$\dots$};
        \node (z2) at (8,2) {};
        \draw (A2) -- (6,-2);
        \draw (A2) -- (x2);
        \draw (A2) -- (z2);
        \node at (9,0) {=};
        \node[fill, circle] (A3) at (13,-1) {};
        \node[fill, circle] (B3) at (11,1) {};
        \node at (11,3) {$\dots$};
        \node (x3) at (13,1) {};
        \node (y3) at (14,1) {$\dots$};
        \node (z3) at (15,1) {};
        \draw (A3) -- (13,-3);
        \draw (A3) -- (B3);
        \draw (A3) -- (x3);
        \draw (A3) -- (z3);
        \node (v) at (10,3) {};
        \node (w) at (12,3) {};
        \draw (B3) -- (v);
        \draw (B3) -- (w);
      \end{tikzpicture}
  \]
\end{definition}

\section{The Boardman-Vogt resolution}

In this section, we recall the Boardman-Vogt resolution~\cite{BV73} in spaces and in posets.

\begin{definition}
  If $P$ is a topological operad, its {\em Boardman-Vogt resolution\/} is an equivalence of operads $W\br{P}\to P$ where $W\br{P}\br{A}$, for some finite set $A$, is defined as follows: first, for a rooted tree $T$ with $A$ as its set of leaves, we consider the space of all labellings of the vertices $v$ of $T$ by points in $x\br{v}\in P\br{\mathrm{in}\br{v}}$, and all of its inner edges $e$ by a {\em length\/} $\lambda\br{e}\in \sbr{0,1}$. The space $W\br{P}\br{A}$ is obtained by quotienting the union of those spaces for all such trees
  \[ \coprod_T\br{\prod_{v}P\br{\mathrm{in}\br{v}}\times \prod_e\sbr{0,1}} \]
  by the equivalence relation generated by isomorphisms of trees, as well as
  \begin{enumerate}
  \item A point with an edge $e$ between two vertices $v_1$ and $v_2$ labelled $0$ is equivalent to the point of the tree obtained by contracting this edge given by the same labels on any edges besides $e$ and all vertices besides $v_1$ and $v_2$ and whose label on the vertex given by the contraction of $v_1$ and $v_2$ is given by $x\br{v_1}\circ_{e}x\br{v_2}$ where $x\br{v_1}$ and $x\br{v_2}$ are the labels of $v_1$ and $v_2$ in the original point:
        
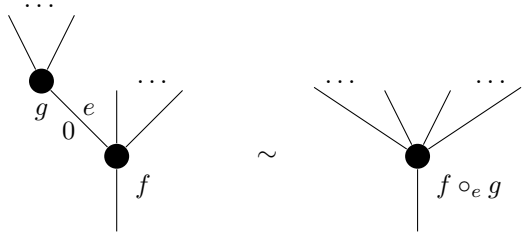
\[
      \begin{tikzpicture}[scale=0.5]
        \node[fill, circle, label=below right:{$f$}] (A) at (0,0) {};
        \node[fill, circle, label=below:{$g$}] (B) at (-2,2) {};
        \node (x) at (0,2) {};
        \node (y) at (1,2) {$\dots$};
        \node (z) at (2,2) {};
        \node (v) at (-3,4) {};
        \node at (-2,4) {$\dots$};
        \node (w) at (-1,4) {};
        \draw (A) -- (0,-2);
        \draw (A) -- (B) node[midway, label={[label distance=-8]-135:$0$}, label={[label distance=-8]45:$e$}] {};
        \draw (A) -- (x);
        \draw (A) -- (z);
        \draw (B) -- (v);
        \draw (B) -- (w);
        \node at (4,0) {$\sim$};
        \node[fill, circle, label=below right:{$f\circ_e g$}] (A2) at (8,0) {};
        \node (x2) at (9,2) {};
        \node (y2) at (10,2) {$\dots$};
        \node (z2) at (11,2) {};
        \node (v2) at (5,2) {};
        \node at (6,2) {$\dots$};
        \node (w2) at (7,2) {};
        \draw (A2) -- (8,-2);
        \draw (A2) -- (x2);
        \draw (A2) -- (z2);
        \draw (A2) -- (v2);
        \draw (A2) -- (w2);
      \end{tikzpicture}
    \]
  \item A point of a tree $T$ containing a vertex $v$ with only a single input edge labelled by the unit $\id$ of the operad $P$ is equivalent to the point in the tree obtained from $T$ by removing $v$ and linking the other vertex of its input edge to the other vertex of its output edge given by the same labels on all other edges and vertices and whose label on the new edge is given by the maximum of the labels on the input and the output edge of $v$: 
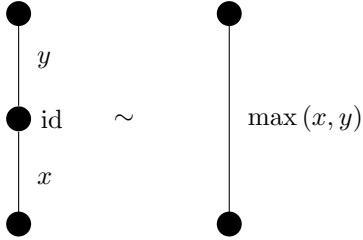
    \[
      \begin{tikzpicture}[scale=0.7]
        \node[fill, circle] at (0,2) {};
        \node[fill, circle] at (0,-2) {}; 
        \node[fill, circle,label=right:{$\id$}] (A) at (0,0) {};
        \draw (A) -- (0,-2) node[midway, label=right:{$x$}] {};
        \draw (A) -- (0,2) node[midway, label=right:{$y$}] {};
        \node at (2,0) {$\sim$};
        \node[fill, circle] at (4,2) {};
        \node[fill, circle] at (4,-2) {};
        \draw (4,-2) -- (4,2) node[midway, label=right:{$\max\br{x,y}$}] {};
      \end{tikzpicture}
    \]
  \end{enumerate}

  The spaces $W\br{P}\br{A}$ have the structure of an operad where the composition is given by grafting the trees and labelling the new inner edge by $1$. The map $W\br{P}\to P$ is given as follows: for a point $x$ given by some labelling of some tree $T$, we first map it to the point $x^\prime$ by changing the labels on all edges to zero. This point is equivalent to a labelling of $C_A$, the corolla on $A$, but such points are completely determined by an element of $P\br{A}$, that we can map $x$ to. Informally, we might think of this map as just forgetting all edge labels and composing all labels on vertices.
\end{definition}

\begin{remark}\label{remBVPosets}
  The Boardman-Vogt resolution also makes sense in categories other than spaces by replacing $\sbr{0,1}$ by some appropriate interval object. In posets, for example, we can take the linear order $\sbr{1}=\cb{0\leq 1}$ and define $W\br{P}$, accordingly.
  In this case, as any edge labelled $0$ can be contracted, any point of $W\br{P}\br{A}$ is represented by a labelling that assigns $1$ to all inner edges (in other words, we may disregard this labelling). Thus, we may represent the elements simply by some tree $T$ all of whose vertices have labels by elements of $P$ other than the unit. These labelled trees are ordered by $\br{S,p}\leq \br{T,q}$, if $S\rightarrowtail T$ and the labelling $p$ is less than or equal to the one given by composing the elements in the labelling $q$ of $T$ along the contracted edges.

  The corresponding map $\pi: W\br{P}\to P$ gives an equivalence of operads --- its inverse is given by $\iota: P\to W\br{P}$ sending an element of $P\br{A}$ to the corolla $C_A$ with the element as the label of the unique vertex (note that $\iota$ is only a map of symmetric sequences, not of operads). For a fixed $A$ and a tree $T$ with $A$ as its set of leaves, we have $C_A\rightarrowtail T$, so that, by definition of the map $\pi$, we have $\iota\circ\pi\leq \id$ and, further, clearly, $\pi\circ\iota=\id$.

  The poset $W\br{P}\br{A}$ also possesses a stratification indexed by trees with $W\br{P}\br{T}$ consisting of the labellings of $T$, and we have
  \[ W\br{P}\br{T}=\prod_{v\in T}P\br{\mathrm{in}\br{v}} \]
\end{remark}

\section{Relating $\FM_n$ to $\calg_n$}

Recall the definition of the maps $\psi:\NConf_n\br{A}\subseteq \Conf_n\br{A}\to \calg_n\br{A}\subseteq \calg_n^{ext}\br{A}$ from~\ref{remGraphsConf}. Clearly, these extend to maps
\[ \FM_n\br{T}=\prod_{v\in T}\NConf_n\br{A}\to \prod_{v\in T}\calg_n\br{A}=W\br{\calg_n}\br{T} \]
and they fit together to define a map $\psi: \FM_n\br{A}\to W\br{\calg_n}\br{A}$ --- this gives a map of operads.

\begin{theorem}\label{thmFMisEn}
  The map $\FM_n\to W\br{\calg_n}$ gives rise to an equivalence of operads between $\FM_n$ and $\calg_n$.
\end{theorem}
\begin{proof}
  We have already seen that $W\br{\calg_n}\to\calg_n$ is an equivalence, so it suffices to see that we obtain an equivalence between $\FM_n$ and $W\br{\calg_n}$. For this, consider the diagram
  \[
    \begin{tikzcd}
      \FM_n\br{A} & \abs{\FM_n\br{A}_{W\br{\calg_n}\br{A}}} \ar[l]\ar[r] & \abs{W\br{\calg_n}\br{A}} \\
      \NConf_n\br{A} \ar[u, hook] \ar[d, hook'] & \abs{\NConf_n\br{A}_{\calg_n\br{A}}}\ar[l]\ar[r]\ar[u, hook] \ar[d, hook'] & \abs{\calg_n\br{A}}\ar[u, hook] \ar[d, symbol={=}] \\
      \Conf_n\br{A} & \abs{\Conf_{n}\br{A}_{\calg_n\br{A}}}\ar[l]\ar[r] & \abs{\calg_n\br{A}} \\
    \end{tikzcd}
  \]
  Here, all the maps in the first and last column are equivalences:
  \begin{itemize}
  \item The map $\NConf_n\br{A}\to \FM_n\br{A}$ is an equivalence as, by Remark~\ref{remStratification}, it is the inclusion of the interior of a manifold with corners.
  \item The map $\NConf_n\br{A}\to \Conf_n\br{A}$ is the equivalence we remarked on in~\ref{defFM}.
  \item The map $\abs{\calg_n\br{A}}\to \abs{W\br{\calg_n}\br{A}}$ is an equivalence, as it is the geometric realisation of the equivalence from Remark~\ref{remBVPosets}.
  \item The map $\abs{\calg_n\br{A}}\to \abs{\calg_n\br{A}}$ is just the identity.
  \end{itemize}
  Further, note that the bottom row is an equivalence by~\ref{propGisEn}. Similarly, or via diagram chase, we obtain that the middle row is one, and so is the top-left map. Thus, we obtain the diagram
  \[
    \begin{tikzcd}
      \FM_n\br{A} & \abs{\FM_n\br{A}_{W\br{\calg_n}\br{A}}} \ar[l,iso']\ar[r] & \abs{W\br{\calg_n}\br{A}} \\
      \NConf_n\br{A} \ar[u, iso, hook] \ar[d, iso', hook'] & \abs{\NConf_n\br{A}_{\calg_n\br{A}}}\ar[l, iso']\ar[r, iso]\ar[u, hook] \ar[d, hook'] & \abs{\calg_n\br{A}}\ar[u, iso', hook] \ar[d,symbol={=}] \\
      \Conf_n\br{A} & \abs{\Conf_n\br{A}_{\calg_n\br{A}}} \ar[l, iso']\ar[r, iso] & \abs{\calg_n\br{A}} \\
    \end{tikzcd}
  \]
  and by its commutativity, all other maps must be equivalences as well, proving the theorem.
\end{proof}

\section{Operads from $n$-fold monoidal categories}

In this section, we will define the operad $\calm_n$ based on $n$-fold monoidal categories introduced by Balteanu, Fiedorowicz, Schwänzl and Vogt in~\cite{BFSV98} and give a direct proof that it is an $E_n$ operad. Our formalism will, however, be quite different from that in~\cite{BFSV98}.

\begin{definition}[The operad $\calm_n$]
  For $n\geq 1$, we define the operad $\calm_n$ in posets as follows: for a finite set $A$, we let $\calm_n\br{A}$ be the equivalence class of planar rooted trees whose set of leaves is $A$ and whose vertices are labelled by the numbers $1,\dots, n$. The equivalence relation is generated by

  \begin{enumerate}
  \item Any two stumps are equivalent.
    
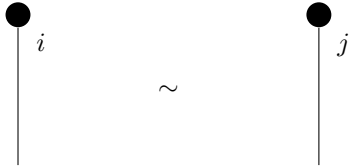
\[
      \begin{tikzpicture}
        \node[fill, circle,label=below right:$i$] (A) at (0,0) {};
        \node[fill, circle,label=below right:$j$] (B) at (4,0) {};
        \draw (A) -- (0,-2);
        \draw (B) -- (4,-2);
        \node at (2, -1) {$\sim$};
      \end{tikzpicture}
    \]
  \item Any stump, other than one directly attached to the root edge, can be removed.
    
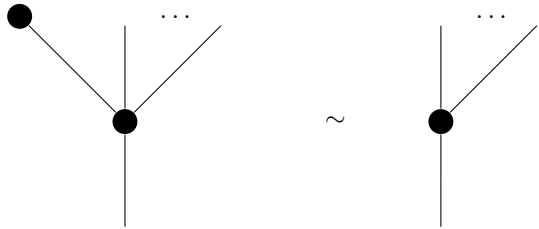
\[
      \begin{tikzpicture}[scale=0.7]
        \node[fill, circle] (A) at (0,0) {};
        \node[fill, circle] (B) at (-2,2) {};
        \node (x) at (0,2) {};
        \node (y) at (1,2) {$\dots$};
        \node (z) at (2,2) {};
        \draw (A) -- (0,-2);
        \draw (A) -- (B);
        \draw (A) -- (x);
        \draw (A) -- (z);
        \node at (4,0) {$\sim$};
        \node[fill, circle] (A2) at (6,0) {};
        \node (x2) at (6,2) {};
        \node (y2) at (7,2) {$\dots$};
        \node (z2) at (8,2) {};
        \draw (A2) -- (6,-2);
        \draw (A2) -- (x2);
        \draw (A2) -- (z2);
      \end{tikzpicture}
    \]
  \item Edges between vertices with the same label can be contracted.
    
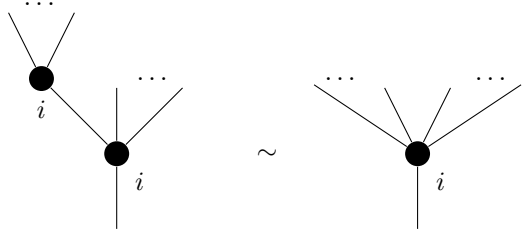
\[
      \begin{tikzpicture}[scale=0.5]
        \node[fill, circle,label=below right:$i$] (A) at (0,0) {};
        \node[fill, circle,label=below:$i$] (B) at (-2,2) {};
        \node (x) at (0,2) {};
        \node (y) at (1,2) {$\dots$};
        \node (z) at (2,2) {};
        \node (v) at (-3,4) {};
        \node at (-2,4) {$\dots$};
        \node (w) at (-1,4) {};
        \draw (A) -- (0,-2);
        \draw (A) -- (B);
        \draw (A) -- (x);
        \draw (A) -- (z);
        \draw (B) -- (v);
        \draw (B) -- (w);
        \node at (4,0) {$\sim$};
        \node[fill, circle, label=below right:$i$] (A2) at (8,0) {};
        \node (x2) at (9,2) {};
        \node (y2) at (10,2) {$\dots$};
        \node (z2) at (11,2) {};
        \node (v2) at (5,2) {};
        \node at (6,2) {$\dots$};
        \node (w2) at (7,2) {};
        \draw (A2) -- (8,-2);
        \draw (A2) -- (x2);
        \draw (A2) -- (z2);
        \draw (A2) -- (v2);
        \draw (A2) -- (w2);
      \end{tikzpicture}
    \]
  \item Vertices with only one input edge can be removed.
    
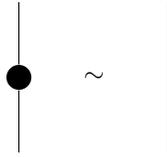
\[
      \begin{tikzpicture}[scale=0.5]
        \node[fill, circle] (A) at (0,0) {};
        \draw (A) -- (0,-2);
        \draw (A) -- (0,2);
        \node at (2,0) {$\sim$};
        \draw (4,-2) -- (4,2);
      \end{tikzpicture}
    \]
  \end{enumerate}

  Further, for two such trees $T_1$ and $T_2$, we define the order by setting $T_1\leq T_2$, if, for any two $a,b\in A$, the label on the lowest vertex of the path from $a$ to $b$ in $T_1$ is less than or equal to that in $T_2$, and the inequality is strict, if the order of $a$ and $b$ (with respect to the planar structure) is different in $T_1$ and $T_2$. Equivalently, this is the first vertex in the intersection between the paths from $a$ and, respectively, $b$ to the root.
  
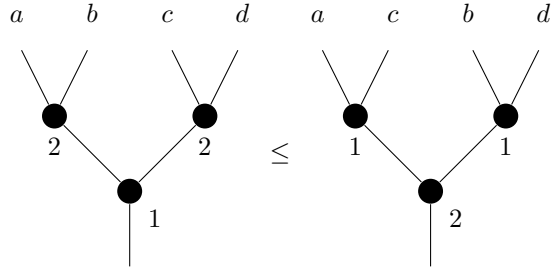
\[
    \begin{tikzpicture}[scale=0.5]
      \node[label=above:$a$] (a) at (-3,4) {};
      \node[label=above:$b$] (b) at (-1,4) {};
      \node[fill, circle, label=below:$2$] (L) at (-2,2) {};
      \node[label=above:$c$] (c) at (1,4) {};
      \node[label=above:$d$] (d) at (3,4) {};
      \node[fill, circle, label=below:$2$] (R) at (2,2) {};
      \node[fill, circle, label=below right:$1$] (M) at (0,0) {};
      \draw (a) -- (L);
      \draw (b) -- (L);
      \draw (c) -- (R);
      \draw (d) -- (R);
      \draw (R) -- (M);
      \draw (L) -- (M);
      \draw (M) -- (0,-2);
      \node at (4,1) {$\leq$};

      \node[label=above:$a$] (a2) at (5,4) {};
      \node[label=above:$c$] (c2) at (7,4) {};
      \node[fill, circle, label=below:$1$] (L2) at (6,2) {};
      \node[label=above:$b$] (b2) at (9,4) {};
      \node[label=above:$d$] (d2) at (11,4) {};
      \node[fill, circle, label=below:$1$] (R2) at (10,2) {};
      \node[fill, circle, label=below right:$2$] (M2) at (8,0) {};
      \draw (a2) -- (L2);
      \draw (c2) -- (L2);
      \draw (b2) -- (R2);
      \draw (d2) -- (R2);
      \draw (R2) -- (M2);
      \draw (L2) -- (M2);
      \draw (M2) -- (8,-2);
    \end{tikzpicture}
  \]

  The operad structure is defined as follows: the composition is given by grafting trees and bijections $B\isoto A$ act by relabelling the vertices. The unit, for $\cb{a}=A$, is given by the tree consisting of only an edge labelled $a$. This makes $\calm_n$ into an operad in posets.

  
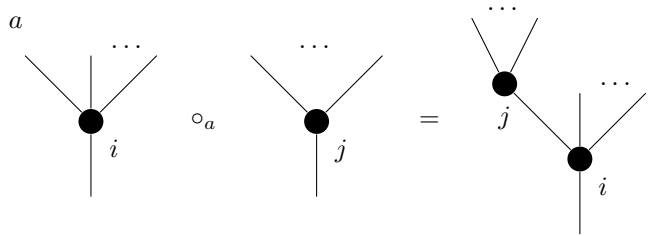
\[
    \begin{tikzpicture}[scale=0.5]
        \node[fill, circle, label=below right:{$i$}] (A) at (0,0) {};
        \node[label=above:$a$] (B) at (-2,2) {};
        \node (x) at (0,2) {};
        \node (y) at (1,2) {$\dots$};
        \node (z) at (2,2) {};
        \draw (A) -- (0,-2);
        \draw (A) -- (B);
        \draw (A) -- (x);
        \draw (A) -- (z);
        \node at (3,0) {$\circ_a$};
        \node[fill, circle, label=below right:{$j$}] (A2) at (6,0) {};
        \node (x2) at (4,2) {};
        \node (y2) at (6,2) {$\dots$};
        \node (z2) at (8,2) {};
        \draw (A2) -- (6,-2);
        \draw (A2) -- (x2);
        \draw (A2) -- (z2);
        \node at (9,0) {=};
        \node[fill, circle, label=below right:{$i$}] (A3) at (13,-1) {};
        \node[fill, circle, label=below:{$j$}] (B3) at (11,1) {};
        \node at (11,3) {$\dots$};
        \node (x3) at (13,1) {};
        \node (y3) at (14,1) {$\dots$};
        \node (z3) at (15,1) {};
        \draw (A3) -- (13,-3);
        \draw (A3) -- (B3);
        \draw (A3) -- (x3);
        \draw (A3) -- (z3);
        \node (v) at (10,3) {};
        \node (w) at (12,3) {};
        \draw (B3) -- (v);
        \draw (B3) -- (w);
      \end{tikzpicture}
  \]

  We will write $\abs{\calm_n}$ for the topological operad obtained by taking the classifying spaces of the posets. We will call this the {\em iterated monoidal operad}.
\end{definition}

\begin{remark}
  Any such equivalence class of trees is canonically represented by a tree that contains no edges between vertices of the same label, vertices with only one input edge, or stumps (unless the tree itself is only a single stump, in which case we also have to choose a canonical label for this stump, say, $1$). We will usually identify the equivalence class with its canonical representative.
\end{remark}

\begin{remark}
  Such a labelled tree corresponds to an object of the $n$-fold monoidal category generated by one object for each element of $A$: inductively, we define the tree consisting of only a leaf (which is, at the same time, the root) labelled $a$ (with $\cb{a}=A$) to correspond to the object $a$. Now, for a tree whose root vertex is labelled $l$ and whose inputs (viewed, again, as rooted trees themselves) correspond to the objects $o_1,\dots,o_k$ (in the order corresponding to the planar structure), we define the corresponding object to be $o_1\square_l \dots \square_l o_k$. In this context, the first two relations on our trees correspond to the fact that we have a common unit, and the last two correspond to (strict) associativity and unitality of the monoidal products.

  
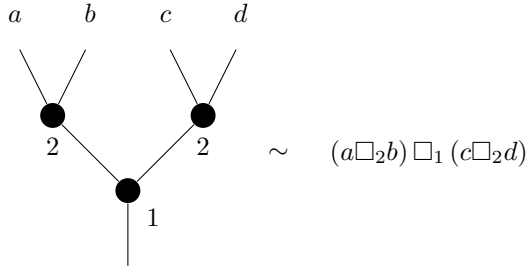
\[
    \begin{tikzpicture}[scale=0.5]
      \node[label=above:$a$] (a) at (-3,4) {};
      \node[label=above:$b$] (b) at (-1,4) {};
      \node[fill, circle, label=below:$2$] (L) at (-2,2) {};
      \node[label=above:$c$] (c) at (1,4) {};
      \node[label=above:$d$] (d) at (3,4) {};
      \node[fill, circle, label=below:$2$] (R) at (2,2) {};
      \node[fill, circle, label=below right:$1$] (M) at (0,0) {};
      \draw (a) -- (L);
      \draw (b) -- (L);
      \draw (c) -- (R);
      \draw (d) -- (R);
      \draw (R) -- (M);
      \draw (L) -- (M);
      \draw (M) -- (0,-2);
      \node at (4,1) {$\sim$};
      \node at (8,1) {$\br{a\square_2 b}\square_1\br{c\square_2 d}$};
    \end{tikzpicture}
  \]
  
  This was the original definition used in~\cite{BFSV98}. They set $\calm_n\br{k}$ to be the free $n$-fold monoidal category on the objects $1,\dots, k$ whose morphisms are generated by the interchange maps of the form \[\br{a_1\square_i \dots \square_i a_m}\square_j\br{b_1\square_i \dots \square_i b_m}\to \br{a_1\square_j b_1}\square_i\dots \square_i \br{a_m\square_j b_m}\] for $j<i$ (in the case $m=2$, this interchange map is simply the example on the previous page). They have shown that this makes the $\calm_n\br{k}$ into a poset whose order is the same as the one we defined on the trees.
\end{remark}

Now, we will relate this operad to the little cubes:

\begin{remark}
  It is well known that the little cubes operad $\calc_n$ is equivalent to that of decomposable configurations (recall that a configuration $c$ of cubes is decomposable, if it consists either of only a single cube, or there is some hyperplane, orthogonal to one of the axes, dividing the configuration into two (non-empty) decomposable configurations). Thus, we may instead work to relate the operad $\cald_n$ of decomposable configurations of little cubes to $\calm_n$. See, for example, Dunn~\cite{Dun88} for reference.
\end{remark}

\begin{definition}  
  We will associate to each tree $T\in \calm_n\br{A}$ a subspace $G\br{T}\subseteq \cald_n\br{A}$ of those configurations possessing a decomposition resembling $T$. More precisely, we set $G\br{T}$ to be the set of such configurations $c$, such that, for each $a,b\in A$ with $a$ to the left of $b$ in the planar structure on $T$ and with the lowest vertex in the path from $a$ to $b$ labelled $i$, we have that $i$-th coordinates of the points of the cube $c_a$ are lower than those of the $i$-th coordinates of the points of the cube $c_b$. Equivalently, $G\br{T}$ is recursively defined as follows: $G\br{\cb{a}}=\cald_n\br{\cb{a}}$ for any singleton $\cb{a}$. Further, for some tree $T$ whose root vertex is labelled $i$ and has, as inputs, the trees $T_1,\dots, T_k$ (in that order according to the planar structure), the set $G\br{T}$ consists of those configurations $c$ such that there are $k-1$ hyperplanes orthogonal to the $i$-th axis dividing the configuration into $k$ smaller configurations $c_1,\dots,c_k$ (in order along the $i$-th axis) with $c_j\in G\br{T_j}$.
  
  Next, we define spaces $F\br{T}$, for each $T\in \calm_n\br{A}$ via $F\br{T}\coloneqq \bigcup_{T^ \prime\leq T}G\br{T^\prime}$. We will also define a map $\varphi: \cald_n\br{A}\to \calm_n\br{A}$ by mapping $c$ to the smallest tree $T$ with $c\in F\br{T}$ --- this gives a lax morphism of operads. Note that $F\br{T}=\varphi^ {-1}\br{\down T}$, but we do not necessarily have $G\br{T}=\varphi^ {-1}\br{T}$.

\end{definition}

The following proof is mostly a direct translation of that of~\cite{BFSV98} to our setting.

\begin{prop}\label{propEqCM}
  The map $\varphi: \cald_n\to \calm_n$ induces an equivalence. In particular, the operad $\calm_n$ is an $E_n$-operad.
\end{prop}
\begin{proof}
  As remarked in~\cite{BFSV98}, the map $\varphi$ is sufficiently cofibrant, since the inclusion of a finite union of convex subspaces of $\R^N$ into a larger such union is a cofibration. By~\ref{corEquivalence}, it thus suffices to show that each $F\br{T}$ is a retract of a CW-complex, and that it is contractible. The first point is clear, as the $G\br{T}$ can be defined in terms of simple, linear inequalities. The second point will be a consequence of the following remark and lemma:

  \begin{remark}
    Any equivalence class of trees with more than one leaf can also be represented by a full binary tree, i.e.~one where all vertices have exactly two input edges (note that this representation is not unique, but it will suffice for our purposes). Further, our ordering on trees, as well as the subspaces $G$ and $F$ depend only on the equivalence class of the trees. In particular, we have \[F\br{T}= \bigcup_{T^ \prime\leq T, T^\prime \text{ is full binary}}G\br{T^\prime}\]
  \end{remark}

  For two full binary trees $T_1$ and $T_2$, as well as $i\in \N$, write $T_1\sim_t T_2$, if the first $t$ levels of $T_1$ and $T_2$ (starting from the root) coincide. More formally, let $T_1^{\br{t}}$ be the subtree of $T_1$ containing all those vertices, edges, and leaves reachable from the root by a path containing at most $t$ edges and define $T_2^{\br{t}}$ similarly, then we ask for the two to coincide.
  
  \begin{lemma}
    For every full binary tree $T$ we have that the inclusion of spaces
    \[ \bigcup_{T^\prime\leq T, T^\prime \sim_{t+1} T}G\br{T^\prime}\subseteq \bigcup_{T^\prime\leq T, T^\prime \sim_{t} T}G\br{T^\prime} \]
    where the union is taken only over full binary trees is a deformation retract.
  \end{lemma}
  \begin{proof}
    Our ordering $\leq$ of trees is preserved by grafting trees (i.e.\ if $T_1\leq T_1^\prime,\dots,T_k\leq T_k^\prime$ and $T$ is some tree, then we have $T\circ \br{T_1,\dots,T_k}\leq T\circ\br{T_1^\prime,\dots, T_k^\prime}$). Therefore, it suffices to show that $\bigcup_{T^\prime\leq T, T^\prime\sim_1 T}G\br{T^\prime}\subseteq F\br{T}$ is a retraction: if $T$ is given by grafting some trees $T_1,\dots, T_k$ onto $T^{\br{t}}$, then $\bigcup_{T^\prime\leq T,T^\prime\sim_t T}G\br{T^\prime}$ is a subspace of $F\br{T_1}\times\cdots\times F\br{T_k}\times G\br{T^{\br{t}}}$, each of the factors of this product but the last retracts to $\bigcup_{T_j^\prime\leq T_j, T_i^\prime\sim_1 T_j}G\br{T_j^\prime}$, the combined retraction stays within the subspace, and ends up in the space $\bigcup_{T^\prime\leq T,T^\prime\sim_{t+1}T}G\br{T^\prime}$ which we view as a subspace of $\bigcup_{T_1^\prime\leq T_1, T_1^\prime\sim_1 T_1}G\br{T^\prime_1}\times \cdots \times \bigcup_{T_k^\prime\leq T_k, T_k^\prime\sim_1 T_k}G\br{T_k^\prime}\times G\br{T^{\br{t}}}$.

    Now, let $T\in\calm_n\br{A}$, and let the root vertex of $T$ be labelled $i$ and have inputs $T_1,T_2$. Further, let $T^\prime\leq T$. If $T^\prime\sim_1T$, there is nothing for us to do, so suppose the root vertex of $T^\prime$ is instead labelled $j$ for some $j<i$ (here, it is important that we use full binary trees, for otherwise it might also be the case that the root vertices have the same label but a different number of inputs). Further, let $A_1, A_2$ be the labels of the leaves of $T_1, T_2$, respectively, and let $T_1^\prime,T_2^\prime$ be the inputs of the root vertex of $T^\prime$. Next, for $l\in \cb{1,2}$, let $h_l:\sbr{0,1}\times\sbr{0,1}\to\sbr{0,1}$ be some homotopy from the identity to the obvious affine map $\sbr{0,1}\to \sbr{\frac{l-1}{2}, \frac{l}{2}}$. We can then define a homotopy $h$ from the inclusion $G\br{T^\prime}\to F\br{T}$ to a map whose image ends up in $\bigcup_{\tilde{T}\leq T, \tilde{T}\sim_1 T}G\br{\tilde{T}}$ as follows:

    We let $h\br{t,c}$ be the configuration obtained by applying $h_l\br{t,\bullet}$ to the $i$-th coordinate of those cubes $c_a$ with $a\in A_l$. This is well defined, as the cubes are separated along the $j$-th coordinate, which is unchanged by $h$, and, at $t=1$, the $i$-th coordinate of the points of $c_a$ with $A_l$ are in $\sbr{\frac{l-1}{2},\frac{l}{2}}$, so we obtain a desired decomposition by cutting along the family of hyperplanes defined by their $i$-th coordinate being one of the $\frac{l}{2}$.

    Stitching those homotopies together gives the desired retraction.
  \end{proof}

  We proceed with the proof of Proposition~\ref{propEqCM}.
  
  Note that for any two trees $T_1,T_2\in \calm_n\br{A}$, we have $T_1\sim_0 T_2$. Moreover, if $h$ is the height of $T_1$, we have $T_1\sim_{h}T_2$, if and only if $T_1=T_2$. Thus, iterating the previous lemma, we see that, for $T\in \calm_n$, the space $G\br{T}$ is a deformation retract of $F\br{T}$. But $G\br{T}$ is defined in terms of linear inequalities so it is convex. Therefore, the only thing left to show is that each $G\br{T}$ is non-empty, which we want to see by induction on the number of leaves of $T$.

  For a tree $T$ with only a single leaf, this is clear. Now, suppose $T$ has a root vertex $v$ with input trees $T_1,\dots,T_k$ (in order of the planar structure) and label $l$, and suppose that $c_1,\dots, c_k$ are configurations with $c_i\in G\br{T_i}$. Further, we define $c_i^\prime$ to be the configuration given by scaling $c_i$ along the $l$-th coordinate to lie in $\sbr{\frac{i-1}{k},\frac{i}{k}}$ and set $c$ to be the union of the $c_i^\prime$, then this configuration $c$ is as required. Therefore, we see that $F\br{T}$ is also contractible which completes the proof.
\end{proof}

\begin{example}
  For the tree $T$ given by
  
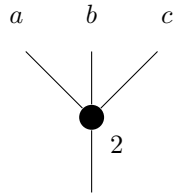
\[
    \begin{tikzpicture}[scale=0.5]
      \node[fill, circle, label=below right:{$2$}] (A) at (0,0) {};
      \node[label=above:$a$] (B) at (-2,2) {};
      \node[label=above:$b$] (x) at (0,2) {};
      \node[label=above:$c$] (z) at (2,2) {};
      \draw (A) -- (0,-2);
      \draw (A) -- (B);
      \draw (A) -- (x);
      \draw (A) -- (z);
    \end{tikzpicture}
  \]
  the retraction $F\br{T}\to \bigcup_{T^\prime\leq T, T^\prime\sim_1 T}G\br{T^\prime}=G\br{T}$ can be pictured as follows:
  
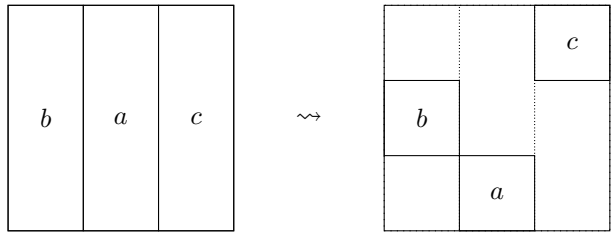
\[
    \begin{tikzpicture}
      \draw (0,0) rectangle (3,-3);
      \draw (0,0) rectangle (1,-3);
      \draw (1,0) rectangle (2,-3);
      \draw (2,0) rectangle (3,-3);
      \node at (0.5,-1.5) {$b$};
      \node at (1.5,-1.5) {$a$};
      \node at (2.5,-1.5) {$c$};
      \node at (4,-1.5) {$\rightsquigarrow$};
      \draw (5,0) rectangle (8,-3);
      \draw (6,-2) rectangle (7,-3);
      \draw (5,-1) rectangle (6,-2);
      \draw (7,0) rectangle (8,-1);
      \draw[dotted] (5,0) rectangle (6,-3);
      \draw[dotted] (6,0) rectangle (7,-3);
      \draw[dotted] (7,0) rectangle (8,-3);
      \node at (6.5,-2.5) {$a$};
      \node at (5.5,-1.5) {$b$};
      \node at (7.5,-0.5) {$c$};
    \end{tikzpicture}
  \]
\end{example}

\subsection{Relation to complete graphs}

From the previous result, we obtain a diagram
\[
  \begin{tikzcd}
    \calc_n \ar[rd] & \cald_n \ar[l, hook'] \ar[r] \ar[d]& \calm_n\\
    & \Conf\br{n} \ar[r] & \calg_n
  \end{tikzcd}
\]
where all arrows induce equivalences. In particular, it follows that $\calm_n$ and $\calg_n$ are (abstractly) equivalent as operads. However, it would be convenient to find a way to complete this diagram with a map
\[ \mu: \calm_n\to \calg_n \]

This, we will seek to construct, as follows:

Given some tree $T\in \calm_n\br{A}$, we can construct a graph on $A$ by defining the edge between $a$ and $b$ to be directed from $a$ to $b$, if $a$ precedes $b$ in $T$ with respect to the planar structure. Further, we label this edge by $i$, if $i$ is the label on the lowest vertex of the path from $a$ to $b$ in $T$. Given the description of the order on both $\calm_n$ and $\calg_n$, it is easy to see that this is a morphism of operads.

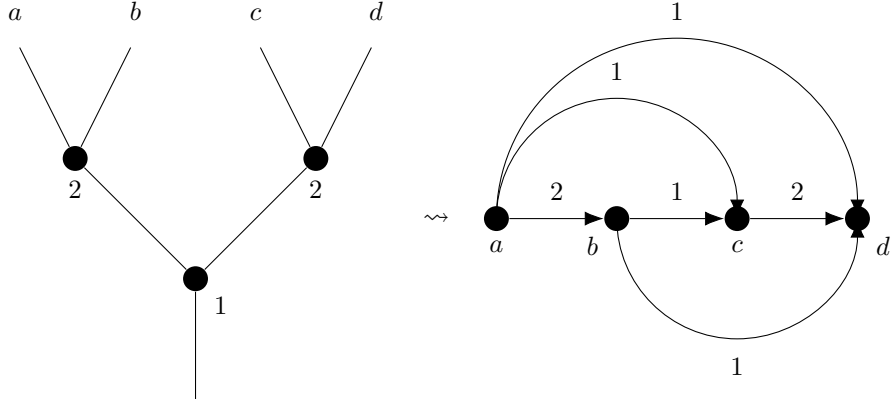
\[
  \begin{tikzpicture}[scale=0.8]
    \node[label=above:$a$] (a) at (-3,4) {};
    \node[label=above:$b$] (b) at (-1,4) {};
    \node[fill, circle, label=below:$2$] (L) at (-2,2) {};
    \node[label=above:$c$] (c) at (1,4) {};
    \node[label=above:$d$] (d) at (3,4) {};
    \node[fill, circle, label=below:$2$] (R) at (2,2) {};
    \node[fill, circle, label=below right:$1$] (M) at (0,0) {};
    \draw (a) -- (L);
    \draw (b) -- (L);
    \draw (c) -- (R);
    \draw (d) -- (R);
    \draw (R) -- (M);
    \draw (L) -- (M);
    \draw (M) -- (0,-2);
    \node at (4,1) {$\rightsquigarrow$};
    \node[circle, fill, label=below:$a$] (a2) at (5,1) {};
    \node[circle, fill, label=below left:$b$] (b2) at (7,1) {};
    \node[circle, fill, label=below:$c$] (c2) at (9,1) {};
    \node[circle, fill, label=below right:$d$] (d2) at (11,1) {};
    \draw[-{Latex[scale=1.5]}] (a2) -- (b2) node[midway, label=above:{$2$}] {};
    \draw[-{Latex[scale=1.5]}] (b2) -- (c2) node[midway, label=above:{$1$}] {};
    \draw[-{Latex[scale=1.5]}] (c2) -- (d2) node[midway, label=above:{$2$}] {};
    \draw[-{Latex[scale=1.5]}] (a2) arc (180:0:3) node[midway,label=above:{$1$}] {}; 
    \draw[-{Latex[scale=1.5]}] (a2) arc (180:0:2) node[midway,label=above:{$1$}] {};
    \draw[-{Latex[scale=1.5]}] (b2) arc (-180:0:2) node[midway,label=below:{$1$}] {};
  \end{tikzpicture}
\]

This map does not make the above diagram commute on the nose, as the configuration of the barycentres might enjoy better separation properties than the configuration of the cubes, themselves. However, we do obtain commutativity up to homotopy (or, rather, the ordering of $\calg_n$). From that we get the following result:
\begin{cor}
  The map $\mu$ induces an equivalence.
\end{cor}
 We hope that there is some combinatorial way to prove this, as well. However, we have not been able to find one (a direct proof via Quillen's theorem A is not possible due to similar counterexamples as for the map $\PP_n\to \calg_n$ above, which can be seen using the next proposition).

For completeness sake, we provide a nice characterisation of the image of $\mu$ --- in particular, this gives an alternative characterisation of the operad $\calm_n$. We will, however, not use this in the remainder of the paper.

\begin{prop}
  Under $\mu$, the elements of $\calm_n$ correspond to those graphs $g\in\calg$ which are decomposable in the following way:

  For all $a,b\in g$ and any path $a\eqqcolon a_0\to a_1\to\dots\to a_k\coloneqq b$, there is an index $i$ such that all edges $a_p\to a_q$ with $p\leq i<q$ have the same weight.
\end{prop}
\begin{proof}
  First, we note that the property is equivalent to the following:

  If the underlying linear order of $g$ is $x_0\to \dots \to x_l$, then $g$ is decomposable in the sense that there exists one edge $x_i\to x_{i+1}$, such that any edge from an element $x_j$, $j\leq i$, to an element $x_k$, $k\geq i+1$, has the same label as the edge $x_i\to x_{i+1}$, and the subgraphs on the vertices $x_0,\dots,x_i$ and $x_{i+1},\dots, x_l$ are both decomposable (technically, this is a recursive definition where we have to add the base case that any graph on a single vertex is decomposable).

  To see this, first assume that $g$ is decomposable and let $a_0\to a_1\to \dots \to a_k$ be any path, then, by iterating the assumption that $g$ is decomposable, we obtain vertices $y_0\to \dots\to y_k$ such that all of the $a_i$ are among these vertices, and such that there is one edge $y_j\to y_{j+1}$ such that any edge from a vertex to the left of or equal to $y_j$ to one to the right of or equal to $y_{j+1}$ has the same weight as the edge $y_j\to y_{j+1}$. In particular, if we let $i$ be the largest index such that $a_i$ is to the left of or equal to $y_j$, then $a_{i+1}$ must be to the right of or equal to $y_{j+1}$, and the same must be true for all $a_p$ and $a_q$ with $p\leq i<q$, respectively. In particular, the edge from $a_p$ to $a_q$ must have the same weight as that $y_j\to y_{j+1}$, so $i$ is an index, as required. The converse follows by repeatedly applying the property to paths of the form $x_p\to x_{p+1}\to \dots \to x_q$ where the underlying linear order of $g$ is $x_0\to\dots\to x_l$.

  Now, we first show that $\mu\br{T}$ is decomposable for any tree $T$. We do this by induction. The base case is clear, so we may move straight to the induction step. For this, let us assume that the root vertex $v$ of $T$ has only two input edges connected to subtrees $T_1$ and $T_2$ where $T_1$ is to the left of $T_2$ in the planar structure. Now, the underlying linear order of $\mu\br{T}$ is of the form $x_0\to\dots\to x_k\to y_0\to\dots\to y_l$ where all the $x_i$ are leaves of $T_1$ and all the $y_j$ are leaves of $T_2$. We claim that the edge $x_k\to y_0$ is as required: any path from $x_i$ to $y_j$ must make use of the root vertex $v$ and (as this is the root vertex) $v$ is the lowest vertex on this path. Thus, the weight on any edge $x_i\to y_j$ must be the label of this root vertex --- in particular, the weights are all equal to that of $x_k\to y_0$, as required. Now, the induction hypothesis applied to $T_1$ and $T_2$ gives that the subgraphs on the $x_0,\dots, x_k$ and $y_0,\dots,y_l$, respectively, are decomposable, finishing the proof of the claim.

  For the converse, suppose we are given a decomposable graph $g$, then we need to find a tree $T$ with $g=\mu\br{T}$. We do this by induction on the number of vertices of $g$. Again, the base case is clear, so we will move to the induction step. Assume the underlying linear order of $g$ is $x_0\to \dots \to x_l$ and let $x_i\to x_{i+1}$ be an edge as in the definition of decomposable and let $k$ be its weight. By induction hypothesis, we find trees $T_1$ and $T_2$, such that $\mu\br{T_1}$ is equal to the subgraph of $g$ on the vertices $x_0,\dots,x_i$ and $\mu\br{T_2}$ is that on $x_{i+1},\dots,x_l$. We define $T$ to be the planar graph with a root vertex $v$ labelled $k$ having two inputs: one to the left connected to $T_1$ and one to the right connected to $T_2$. We claim that $\mu\br{T}=g$. For this, we only have to see that both agree on edges of the form $x_p\to x_q$ with $p\leq i<q$: in $\mu\br{T}$, all such edges are labelled $k$, as the path between the corresponding leaves must go through the root vertex, which is labelled $k$, but this is also their weight in $g$, as, by the definition of decomposability, their weight must be equal to that of the edge $x_i\to x_{i+1}$, which is how we defined $k$.
\end{proof}

\section{The McClure-Smith Operad}

The operad we define in this section is originally due to McClure and Smith~\cite{MS04}. 

\begin{definition}[Lattice paths]
  For a finite set $A$ and a natural number $n$, let $\call\br{A;n}$ be the set of isomorphism classes of diagrams of the form
  \[ A\overset{f}{\leftarrow} L \overset{\alpha}{\to}\sbr{n} \]
  where $L$ is a finite linear order, f is a surjective map of sets, and $\alpha$ is a map of linear orders into $\sbr{n}=\cb{0<1<\cdots <n}$. We will refer to such diagrams as {\em lattice paths}.
  Isomorphisms between two diagrams $A\overset{f}{\leftarrow} L \overset{\alpha}{\to}\sbr{n}$ and $A\overset{g}{\leftarrow} M \overset{\beta}{\to}\sbr{n}$ consist of isomorphisms of linear orders $\theta: L\to M$, such that the diagram
  \[
    \begin{tikzcd}
      L\ar[d, "f"'] \ar[r, "\alpha"] \ar[rd, iso', "\theta"] & \sbr{n}\\
      A & M \ar[l, "g"] \ar[u, "\beta"']
    \end{tikzcd}
  \]
  commutes.
\end{definition}

\begin{definition}\label{defRestr} Later on, it will be convenient to also have a restriction operation for lattice paths: if we are given some subset $U\subseteq A$, then we set
  \[ \bullet|_{U}: \call\br{A;n}\to \call\br{U;n} \]
  to be the map defined by sending some diagram $A\overset{f}{\leftarrow}L\overset{\alpha}{\to}\sbr{n}$ to the diagram \[U\xleftarrow{f|_{f^{-1}\br{U}}}f^{-1}\br{U}\xrightarrow{\alpha|_{f^{-1}\br{U}}}\sbr{n}\]
\end{definition}

\begin{remark}
  \begin{enumerate}[itemindent=*,leftmargin=0pt]
  \item For set theoretic reasons, we have to restrict this definition to linear orders $L$ coming from some convenient universe/level of the von Neumann hierarchy.
  \item If two lattice paths are isomorphic, then the isomorphism between them is unique.
  \end{enumerate}
\end{remark}

\begin{remark}
  The name lattice path comes from the following observation: if we are given a point $A\overset{f}{\leftarrow} L \overset{\alpha}{\rightarrow} \sbr{n}$ with $L=\cb{l_0<\cdots <l_k}$, and if the smallest element of $L$ in the preimage of $i$ under $\alpha$ is $l_{k\br{i}}$, then we can identify this with a subdivided string of symbols \[f\br{l_0}\dots f\br{l_{k\br{1}-1}}\vert f\br{l_{k\br{1}}}\dots f\br{l_{k\br{2}-1}}\vert\qquad \dots \qquad\vert f\br{l_{k\br{n}}}\dots f\br{l_k} \]
  Further, if $A=\cb{a_1,\dots, a_m}$, then each such string corresponds to a path through the $m$-dimensional lattice of size $\abs{f^{-1}\br{a_1}}\times\cdots\times \abs{f^{-1}\br{a_m}}$ subdivided into $n+1$ pieces: we interpret each occurrence of $a_i$ in the string of symbols as taking one step in the $i$-th direction.
  \[ ab|a|bb \quad \rightsquigarrow \quad
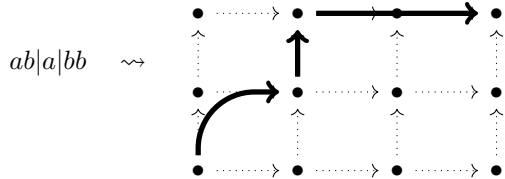
    \begin{tikzcd}[anchor=base, baseline=2.2cm]
      \bullet \ar[r, dotted] & \bullet \ar[r, dotted] \ar[rr, line width=2, -{to[scale=0.25]}]& \bullet \ar[r,dotted] & \bullet\\
      \bullet \ar[r, dotted] \ar[u, dotted] & \bullet \ar[r, dotted]\ar[u, dotted] \ar[u,line width=2, -{to[scale=0.25]}] & \bullet \ar[r,dotted]\ar[u, dotted] & \bullet\ar[u, dotted]\\
      \bullet \ar[r, dotted] \ar[u, dotted] \ar[ur, rounded corners=0.75cm, to path={(\tikztostart.north) |- (\tikztotarget)},line width=2, -{to[scale=0.25]}]& \bullet \ar[r, dotted]\ar[u, dotted] & \bullet \ar[r,dotted]\ar[u, dotted] & \bullet\ar[u, dotted]\\
    \end{tikzcd}
  \]

  The idea of considering these lattice paths and using them to construct the McClure-Smith operad comes from Batanin and Berger~\cite{BB09}.
\end{remark}

\begin{remark}
  $\call\br{A;\bullet}$ has an obvious cosimplicial structure. It also has the structure of an $A$-fold simplicial object, for which we will need some additional notation.  
\end{remark}

\begin{definition}
  For $p: A\to \N$, we let $\call\br{A;n}_p$ be the set of diagrams $A\overset{f}{\leftarrow} L \overset{\alpha}{\rightarrow}\sbr{n}$  such that $f^ {-1}\br{a}\cong \sbr{p\br{a}}$ for each $a\in A$ as a linear order. As this isomorphism is unique, we will identify $f^ {-1}\br{a}$ with $\sbr{p\br{a}}$. Of course every element of $\call\br{A;n}$ lies in $\call\br{A;n}_p$ for some $p$.

  Further, suppose we are given some $a\in A$, as well as $m\in \N$, then we will write $\call\br{A;n}_{a\mapsto m}$ for the union of the $\call\br{A;n}_p$ taken over all $p$ with $p\br{a}=m$ --- this is just the set of those diagrams $A\overset{f}{\leftarrow} L \overset{\alpha}{\to}\sbr{n}$ with $f^{-1}\br{a}\cong \sbr{m}$. We will refer to its elements as {\em $m$-simplices in the $a$-direction}. Now, if we are given a map $\rho:\sbr{k}\to \sbr{m}$, then we can define a map
  \[ \rho^\ast: \call\br{A;n}_{a\mapsto m}\to \call\br{A;n}_{a\mapsto k} \]
  as follows:

  given some diagram $A\overset{f}{\leftarrow}L\overset{\alpha}{\to} \sbr{n}$ in $\call\br{A;n}_p$ with $p\br{a}=m$ (here, we will view $\sbr{m}$ as a subset of $L$ via the isomorphism $\sbr{m}\cong f^{-1}\br{a}$), we define a new diagram \[\rho^\ast\br{A\overset{f}{\leftarrow}L\overset{\alpha}{\to} \sbr{n}}=A\xleftarrow{f\sbr{f\rho/\sbr{m}}}L\sbr{\;\sbr{k}/\sbr{m}}\xrightarrow{\alpha\sbr{\alpha\rho/\sbr{m}}}\sbr{n}\] given by replacing $\sbr{m}$ by $\sbr{k}$ according to $\rho$. More precisely, $L\sbr{\;\sbr{k}/\sbr{m}}$ has as its underlying set the set $L\setminus \sbr{m}\amalg \sbr{k}$, ordered by
  \[
    x\leq y\colonLeftrightarrow
    \begin{cases}
      x\leq y,& x,y\in L\setminus \sbr{m}\\
      x\leq y,& x,y\in \sbr{k}\\
      x\leq \rho\br{y},& x\in L\setminus\sbr{m}, y\in \sbr{k}\\
      \rho\br{x}\leq y,& x\in \sbr{k}, y\in L\setminus\sbr{m}
    \end{cases}
  \]
  and, further, the maps $f\sbr{f\rho/\sbr{m}}$ and $\alpha\sbr{\alpha\rho/\sbr{m}}$ are defined via
  \begin{align*}
    &f\sbr{f\rho/\sbr{m}}\br{x}=f\br{x},& \alpha\sbr{\alpha\rho/\sbr{m}}\br{x}=\alpha\br{x}
  \end{align*}
   for $x\in L\setminus\sbr{m}$ and \begin{align*}&f\sbr{f\rho/\sbr{m}}\br{y}=f\br{\rho\br{y}},&\alpha\sbr{\alpha\rho/\sbr{m}}\br{y}=\alpha\br{\rho\br{y}}\end{align*} for $y\in \sbr{k}$.

  This gives the sets $\br{\call\br{A;n}_{a\mapsto m}}_{m\in \N}$ the structure of a simplicial set, which we will denote by $\call\br{A;n}_{a\mapsto\bullet}$ and refer to as the {\em simplicial set in the $a$-direction}.
\end{definition}

\begin{remark}
  A diagram of the form $A\leftarrow L \rightarrow \sbr{n}\in \call\br{A;n}_{a\mapsto m}$ is a face (more precisely: iterated composition of faces) of another diagram $A\leftarrow M \rightarrow \sbr{n}$ in the $a$-direction, if there is a commutative diagram of the form
  \[
    \begin{tikzcd}
      L\ar[d, "f"'] \ar[r, "\alpha"] \ar[rd, hook, "\theta"] & \sbr{n}\\
      A & M \ar[l, "g"] \ar[u, "\beta"']
    \end{tikzcd}
  \]
  with $\theta$ injective and an isomorphism outside of $f^{-1}\br{a}$. Further, if there exists a diagram $A\leftarrow M \rightarrow \sbr{n}$ and a commutative diagram of the form
  \[
    \begin{tikzcd}
      L\ar[d, "f"'] \ar[r, "\alpha"] \ar[rd, two heads, "\theta"] & \sbr{n}\\
      A & M \ar[l, "g"] \ar[u, "\beta"']
    \end{tikzcd}
  \]
  with $\theta$ surjective and an isomorphism outside of $f^{-1}\br{a}$, then $A\leftarrow L \rightarrow \sbr{n}$ is degenerate in the $a$-direction.
  
\end{remark}

We now begin to use this multi-simplicial and cosimplicial structure to construct an operad in spaces from the $\call\br{A;n}$. For this, some of the details of the constructions can be found in the appendix.

\begin{definition}
  By geometrically realising $\call\br{A;n}$ in all simplicial directions (i.e.\ $\call\br{A;n}_{a\mapsto \bullet}$ for all $a$), we obtain a space that we will denote by $\abs{\call\br{A;n}}$ and, taken together for all $n\in \N$, this gives a cosimplicial space $\abs{\call\br{A;\bullet}}$. We can further take the {\em totalisation\/} of this cosimplicial space to obtain a single, topological space
  \[ \Tot{\abs{\call\br{A;\bullet}}}\coloneqq \Top^\Delta\br{\abs{\Delta^\bullet}, \abs{\call\br{A;\bullet}}} \]
  Suggestively, we will already denote these spaces by $\MS\br{A}$.

\end{definition}

\begin{remark}
  A point in the geometric realisation $\abs{\call\br{A;n}}$ is represented by a pair \[\br{A\overset{f}{\leftarrow}L\overset{\alpha}{\to}\sbr{n}}\otimes \br{t\br{a}}_{a\in A}\] where $t\br{a}\in \Delta^{\abs{f^{-1}\br{a}}-1}$. We will often identify the $\br{t\br{a}}_{a\in A}$ with a family $\br{t_l}_{l\in L}$ with $0\leq t_l\leq 1$ such that if $f^{-1}\br{a}=\cb{l_0,\dots,l_k}$ and $l_0<\cdots<l_k$, then $t\br{a}=\br{t_{l_0},\dots,t_{l_k}}$. In particular, we have $\sum_{i=0}^k t_{l_i}=1$.
  This representation is unique, if we take $A\leftarrow L\rightarrow \sbr{n}$ to be non-degenerate and each $t\br{a}$ to be an interior point.  If $A\leftarrow L\to \sbr{n}$ is degenerate, i.e.\ we have a diagram
  \[
    \begin{tikzcd}
      L\ar[d, "f"'] \ar[r, "\alpha"] \ar[rd, two heads, "\theta"] & \sbr{n}\\
      A & M \ar[l, "g"] \ar[u, "\beta"']
    \end{tikzcd}
  \]
  then the point $\br{A\leftarrow L \to \sbr{n}}\otimes \br{t_l}_{l\in L}$ is the same as the point $\br{A\leftarrow M \to \sbr{n}}\otimes \br{s_m}_{m\in M}$ with $s_m\coloneqq \sum_{l: \theta\br{l}=m}t_l$.
\end{remark}

\begin{definition}\label{defSubstitution}
  We can define a substitution operation for lattice paths. If we are given two elements $A\overset{f}{\leftarrow}L\overset{\alpha}{\to}\sbr{n}$ and $B\overset{g}{\leftarrow}M\overset{\beta}{\to}\sbr{m}$, and if $\sbr{m}\cong f^{-1}\br{a}$ for some $a\in A$, then we can replace each $l\in f^{-1}\br{a}$ by the corresponding fibre of $\beta$ to obtain a new linear order $L\circ_a M$. Further, there are maps
  \[
    \begin{tikzcd}
      L\circ_a M \ar[d, "f\circ_a g"'] \ar[r] & L \ar[r] \ar[d] & \sbr{n} \\
      A\sbr{B/a}\ar[r] & A &
    \end{tikzcd}
  \]
  making $A\sbr{B/a}\leftarrow L\circ_a M \to \sbr{n}$ an element of $\call\br{A\sbr{B/a};\sbr{n}}$. In the cosimplicial structure, this element has the same degree as $A\leftarrow L\rightarrow \sbr{n}$. In the simplicial structure in the $x$-direction for $x\in A\setminus\cb{a}$, it has the same degree as $A\leftarrow L\rightarrow\sbr{n}$, and in the $b$-direction for $b\in B$, it has the same degree as $B\leftarrow M\rightarrow \sbr{m}$.
\end{definition}

\begin{remark}\label{remSubstitution}
  Let $A\overset{f}{\leftarrow}L\overset{\alpha}{\to}\sbr{n}$, $a\in A$, and $B\overset{g}{\leftarrow}M\overset{\beta}{\to}\sbr{m}$ with $\sbr{m}\cong f^{-1}\br{a}$ be as above, and assume we are given a map $\rho: \sbr{m^\prime}\to \sbr{m}$ in $\Delta$. The compositions
  \[\rho^{\ast}\br{A\overset{f}{\leftarrow}L\overset{\alpha}{\to}\sbr{n}}\circ_a \br{B\overset{g}{\leftarrow}M\overset{\beta}{\to}\sbr{m^\prime}} \text{ and } \br{A\overset{f}{\leftarrow}L\overset{\alpha}{\to}\sbr{n}}\circ_a \rho_\ast\br{B\overset{g}{\leftarrow}M\overset{\beta}{\to}\sbr{m^\prime}} \]
  where $\rho^\ast$ is the simplicial, and $\rho_\ast$ the cosimplicial action, coincide.
\end{remark}

\begin{remark}
  The substitution operations induce maps
  \[ \circ_a: \MS\br{A}\times \MS\br{B}\to \MS\br{A\sbr{B/a}} \]
  Given maps $\varphi: \Delta^\bullet\to \abs{\call\br{A;\bullet}}$ and $\psi: \Delta^\bullet\to \abs{\call\br{B;\bullet}}$, we can construct a map $\varphi\circ_a\psi: \Delta^\bullet \to \abs{\call\br{A\sbr{B/a}}}$ as follows: suppose $t=\br{t_0,\dots,t_n}\in \Delta^n$ and $\varphi\br{t}$ is represented by $\br{A\overset{f}{\leftarrow} L\to\sbr{n}}\otimes \br{s_l}_{l\in L}$. Then in particular $s_a=\br{s_l}_{l\in f^{-1}\br{a}}$ is a point in $\Delta^m$ where $\sbr{m}\cong f^{-1}\br{a}$, so suppose $\psi\br{s_a}$ is represented by $\br{B\leftarrow M\rightarrow \sbr{m}}\otimes \br{r_x}_{x\in M}$. Then we let $\br{\varphi\circ_a\psi}\br{t}$ be represented by
  \[ \br{\;\br{A\leftarrow L\to \sbr{n}}\circ_a \br{B\leftarrow M\to \sbr{m}}}\otimes \br{s\circ_a r} \] with $s\circ_a r$ defined by $\br{s\circ_a r}_l=s_l$ for $l\in L\setminus\cb{a}$ and $\br{s\circ_a r}_{x}=r_x$ for $x\in M$. This makes the $\MS\br{A}$ into the McClure-Smith operad introduced in~\cite{MS04}.
\end{remark}

\begin{definition}\label{defWeightLP}
  For a lattice path $x=\br{A\overset{f}{\leftarrow}L\rightarrow \sbr{n}}$, as well as $a\neq b\in A$, we define the {\em weight of $x$ between $a$ and $b$}, $w_x\br{a,b}$, as follows: let $L|_{\cb{a,b}}\coloneqq f^{-1}\br{a}\cup f^{-1}\br{b}$ and define an equivalence relation $\sim$ on $L|_{\cb{a,b}}$ via $l\sim l^\prime$, for $l<l^\prime$, if $f$ is constant on $\cb{m\in L|_{\cb{a,b}}\;\vert\;l\leq m\leq l^\prime}$ and set $w_x\br{a,b}\coloneqq \abs{L|_{\cb{a,b}}/\sim}-1$. In other words, if we imagine $x$ to be represented by a (subdivided) string of symbols in $A$, and let $x^\prime$ be the string obtained from $x$ by deleting all symbols other than $a$ and $b$, then $w_x\br{a,b}$ is the number of times that the string $x^\prime$ switches between $a$ and $b$.

  Further, we define the {\em weight $w\br{x}$ of x\/} to be the maximum of the weights between any $a$ and $b$, $w\br{x}\coloneqq \underset{a,b}{\max}\;w_x\br{a,b}$.

  For each finite set $A$ and each $n,m \in \N$ this allows us to define a subset
  \[ \call_m\br{A;n}\coloneqq \cb{x\in \call\br{A;n}\;\vert\; w\br{x}\leq m} \]
  These sets are clearly compatible with the cosimplicial operations (the maps $\alpha: L\to M$ do not even appear in the definition of the weights) and the simplicial operations (they can only decrease the number of equivalence classes of $L|_{\cb{a,b}}$) and so we get subspaces
  \[ \MS_m\br{A}=\Top^\Delta\br{\Delta^\bullet,\abs{\call_m\br{A;\bullet}}}\subseteq \MS\br{A} \]
  It is, easy to verify that these subspaces are closed under the substitution operations. Thus, this defines a suboperad $\MS_m$ of $\MS$.
\end{definition}

\section{The McClure-Smith Operad and Complete Graphs}\label{secMS}
In this section, we wish to relate the McClure-Smith operad to the complete graphs operad. Before we do so, we need to study some technical details about the geometric realisation and totalisation.

The following lemma is due to McClure and Smith~\cite{MS04}.

\begin{lemma}\label{lemLPCosimplicial}
  For each finite set $A$, there is an isomorphism of cosimplicial spaces 
  \[ \abs{\call\br{A;\bullet}}\cong \abs{\call\br{A;0}}\times \Delta^\bullet \]
  i.e.\ for each $n$, we have homeomorphisms
  \[ \abs{\call\br{A;n}}\cong \abs{\call\br{A;0}}\times \Delta^n \]
  compatible with the cosimplicial structure.
\end{lemma}
\begin{proof}
  For a fixed $A$, the simplex $\Delta^n$ is homeomorphic to the blown up simplex \[\tilde{\Delta}^n\coloneqq \cb{\br{s_0,\dots,s_n}\;\Bigg\vert\; \sum_i s_i=\abs{A}}\] and it will be more convenient to work with these spaces, instead.

  Recall that a point of $\abs{\call\br{A;n}}$ is given by an equivalence class of objects of the form $\br{A\overset{f}{\leftarrow}L\overset{\alpha}{\to}\sbr{n}}\otimes \br{t_l}_{l\in L}$. Now, we can obtain a point in $\tilde{\Delta}^n$ from the $t_l$ by summing over the fibres of $\alpha$. More precisely, we set $s_i\coloneqq \sum_{l\in\alpha^{-1}\br{i}}t_l$ to get the point $\br{s_0,\dots,s_n}$. This defines a map
  \begin{align*}
    &\abs{\call\br{A;n}}\to \abs{\call\br{A;0}}\times \tilde{\Delta}^n\\
    &\br{A\overset{f}{\leftarrow}L\overset{\alpha}{\to}\sbr{n}}\otimes \br{t_l}_{l\in L}\mapsto \br{\br{A\overset{f}{\leftarrow}L\to\sbr{0}}\otimes \br{t_l}_{l\in L}, \br{s_i}_{i\in\sbr{n}}}
  \end{align*}
  To see that this defines a homeomorphism, we wish to construct an inverse.
  
  Suppose we are given a point $\br{\br{A\overset{f}{\leftarrow}L\to\sbr{0}}\otimes \br{t_l}_{l\in L}, \br{s_i}_{i\in\sbr{n}}}$. We wish to interpret the $\br{s_i}_{i\in \sbr{n}}$ as encoding a map $L\to\sbr{n}$. That is, we would like to define a map $\alpha: L\to \sbr{n}$, such that $s_i=\sum_{l\in \alpha^{-1}\br{i}}t_l$. Unfortunately, this is not always possible --- however, we can find a point in the same equivalence class for which this works: as $\sum_l t_l=\abs{A}=\sum_{i}s_i$, there will be unique $l_i\in L$, such that
  \[ \sum_{l<l_i}t_l<\sum_{j\leq i}s_i\leq \sum_{l\leq l_i}t_l \]
  for each $i\in \sbr{n}$.

  Now, let $L^\prime$ be the linear order on the set \[L\setminus\cb{l_i\;\vert\; i\in \sbr{n}}\cup \cb{l^-_i\;\vert\; i\in \sbr{n}}\cup\cb{l^+_i\;\vert\; i\in \sbr{n}}\]
  ordered via
  \[ l<m\colonLeftrightarrow
    \begin{cases}
      l<m, & l,m\in L\\
      l<l_i, & l\in L, m=l^+_i\text{ or }m=l^-_i\\
      l_i<m, & m\in L, l=l^+_i\text{ or }l=l^-_i\\
      l_i<l_j, & l=l_i^+\text{ or } l=l_i^-\text{ and } m=l_j^+\text{ or }m=l_j^-\\
      l=l_i^-\text{ and }m=l_i^+
    \end{cases}
  \]
  In other words, we split each $l_i$ into two elements, $l_i^{-}$ and $l_i^+$ with $l_i^-<l_i^+$ and leave the ordering between all other elements unchanged.
  There is an obvious map $\theta: L^\prime\to L$ sending $l_i^-$ and $l_i^+$ to $l_i$ and any other elements $l$ of $L^\prime$ to $l\in L$, and we set $f^\prime\coloneqq f\circ\theta$. Further, we can define a point $\br{t^\prime_l}_{l\in L^\prime}$ via $t^\prime_l\coloneqq t_l$ for $l\notin \cb{l_i^-,l_i^+\;\vert\; i\in \sbr{n}}$,  $t^\prime_{l_i^-}\coloneqq \sum_{j\leq i}s_j-\sum_{l<l_i}t_l$, and $t^\prime_{l_i^+}\coloneqq \sum_{l\leq l_i}t_l-\sum_{j\leq i}s_j$. Then, $\br{A\leftarrow L^\prime\rightarrow \sbr{0}}\otimes \br{t_l^\prime}_{l\in L^\prime}$ represents the same point as $\br{A\leftarrow L\rightarrow \sbr{0}}\otimes \br{t_l}_{l\in L}$ as is witnessed by the diagram
  \[
    \begin{tikzcd}
      L^\prime \ar[rd, two heads, "\theta"] \ar[d,"f^\prime"'] \ar[r] & \sbr{0} \\
      A & L\ar[l,"f"'] \ar[u]
    \end{tikzcd}
  \]
  Now, if we set $\alpha: L^\prime\to \sbr{n}$ to be the map $\alpha\br{l}=i$ if $\sum_{j<i}s_j<\sum_{m\leq l}t_m\leq \sum_{j\leq i}s_j$, then we have $s_i=\sum_{l\in \alpha^{-1}\br{i}}t^\prime_l$ and so $\br{A\overset{f^\prime}{\leftarrow} L^\prime \overset{\alpha}{\rightarrow}\sbr{n}}\otimes \br{t^\prime_l}_{l\in L^\prime}$ is the desired preimage.
\end{proof}

\begin{definition}
  Recall the definition of the weights $w_x\br{a,b}$ from Definition~\ref{defWeightLP}. We can use these to define a morphism
  \[
    \gamma: \call_m\br{A;n}\to \calg_m\br{A}
  \]
  by sending a lattice path $x=\br{A\overset{f}{\leftarrow}L\overset{\alpha}{\to}\sbr{n}}$ to the graph on $A$ having an edge $a\overset{i}{\to}b$, if the smallest $l$ with $f\br{l}=a$ is smaller than the smallest $m$ with $f\br{m}=b$ and if $w_x\br{a,b}=i$ --- if we think of $x$ as a string of symbols in $A$, the first part simply asks that $a$ occurs before $b$ does.

  We already observed that the definition of the weights is compatible with the simplicial and cosimplicial structure, and so $\gamma$ induces a map
  \[ \gamma: \MS_m\br{A}\to\calg_m\br{A} \]
\end{definition}

\begin{remark}
  $\gamma$ defines a lax morphism of operads: if we have some $a\in A$ and composable lattice paths $A\overset{f}{\leftarrow}L\overset{\alpha}{\to}\sbr{n}$ and $B\overset{g}{\leftarrow}M\overset{\beta}\to\sbr{m}$, and if $x,y\in A\setminus\cb{a}$, then the weights between $x$ and $y$ in the composition is equal to that in the first lattice path, as $L\circ_a M|_{\cb{x,y}}$ depends only on $L$. Similarly, if $x,y\in B$, then the weight between $x$ and $y$ in the composition is equal to that in the second lattice path, as $L\circ_a M|_{\cb{x,y}}$ depends only on $M$. Finally, for $x\in A\setminus\cb{a}$ and $y\in B$, the weight between $x$ and $y$ in the composition is at most that between $x$ and $a$ in the first lattice path, as there is an obvious injection from the equivalence classes of $L\circ_a M|_{\cb{x,y}}$ to those of $L|_{\cb{x,a}}$. Further, the only way that $x$ can occur before $y$ in the composition, if $a$ occurs before $x$ in the first lattice path, is if this injection is {\em not\/} also a bijection --- in other words, the direction of the edge between $x$ and $y$ in $\gamma$ applied to the composition can only differ from that between $x$ and $a$ in $\gamma$ applied to the first lattice path, if the weight between the former is strictly smaller than the weight between the latter.
\end{remark}

\begin{prop}
  For each $g\in \calg_n\br{A}$, the space $\abs{\call_g\br{A;n}}\coloneqq\abs{\gamma^{-1}\br{\down g}}$ is contractible.
\end{prop}
\begin{proof}
  First, note that it is non-empty: if the underlying linear order of $g$ is $a_0<\cdots <a_k$, then the string $a_0a_1\dots a_k\in \call_n\br{A;0}$, together with some coordinates, gives a point of it.

  Now, by~\ref{lemLPCosimplicial}, it suffices to see that the preimage of $g$ in $\abs{\call_n\br{A;0}}$ is contractible. For convenience, we will drop the maps $L\to\sbr{0}$ when denoting elements of $\call_g\br{A;0}$. We will proceed by induction on $\abs{A}$. If $A=\cb{a}$, then the claim is obvious. Now, for the induction step, let us fix $g$ and let $a_0$ be the smallest element of $A$ according to the underlying order of $g$. First of all, we note that it suffices to show that the simplicial set obtained from $\call_g\br{A;0}$ by fixing the simplicial degree in all directions other than $a_0$ is contractible, that is, the simplicial set
  \[X_y\coloneqq\cb{x=\br{A\overset{f}{\leftarrow}L}\in \call_g\br{A;0}\;\Big\vert\; x|_{A\setminus\cb{a_0}}=y }\]
  for some fixed $y\in \call\br{A\setminus\cb{a_0};0}$ with $\gamma\br{y}\leq g|_{A\setminus\cb{a_0}}$. To see this, note that this will give us that the space $\abs{\call_g\br{A;0}}$ is homotopy equivalent to the space $\abs{\call_{g|_{A\setminus\cb{a_0}}}\br{A\setminus\cb{a_0};0}}$ and we may proceed with the induction hypothesis.

  Let us fix one such diagram $y=A\overset{f}{\leftarrow}L\in \call\br{A\setminus\cb{a_0};0}$. A $k$-simplex $M\overset{e}{\to}A$ of $X_y$ gives rise to, and can be identified with, a pullback diagram
  \[
    \begin{tikzcd}
      L\ar[r, hook] \ar[d,"f"'] & M \ar[d,"e"] \\
      A\setminus\cb{a_0} \ar[r, hook] & A
    \end{tikzcd}
  \]
  such that $e^{-1}\br{a_0}\cong\sbr{k}$ and $\gamma\br{e}\leq g$. We can identify the map $e$ with a monotone map $\tilde{e}:\sbr{k}\to \cb{-\infty}\amalg L$ denoting that the $\br{i+1}$-st element of $e^{-1}\br{a_0}$ in $M$ occurs just after $\tilde{e}\br{i}$. More precisely, we set \[\tilde{e}\br{i}\coloneqq \max\cb{l\in L\;\vert\; \abs{e^{-1}\br{a_0}\cap \cb{m\in M\;\vert\; m\leq l}}\leq i}\] where we use the convention that $\max\emptyset=-\infty$. The inverse of this operation is given by setting $M\coloneqq L\amalg\sbr{k}$ using the linear ordering with $l\leq m$, for $l\in L$ and $m\in \sbr{k}$, if $l\leq\tilde{e}\br{m}$.
  The simplicial set $X_y$ has a canonical vertex $v$ given by the map $\tilde{e}:\sbr{0}\to \cb{-\infty}\amalg L$ with $\tilde{e}\br{0}=-\infty$. Again (as in the proof of Lemma~\ref{lemBEDownContractible}), we consider the simplicial set $v/X_y$ whose $k$-simplices are given by
  \[ \br{v/X_y}_k\coloneqq \cb{x\in \sbr{X_y}_{k+1}\;\vert\; x_0=v} \]
  This simplicial set is clearly contractible. Thus, it suffices to see that the projection $\pi: v/X_y\to X_y$ sending a $k$-simplex $x$ (i.e.\ a certain $k+1$-simplex of $X_y$) to the simplex given by forgetting the first vertex has some section $\sigma$. Indeed, if this is the case, then $X_y$ is a retract of the contractible space $v/X_y$ and, hence, contractible. If we are given a $k$-simplex $x=\br{A\overset{e}{\leftarrow}M}$ in $X_y$, then we set $\sigma\br{x}\coloneqq \br{A\xleftarrow{a_0\amalg e}\ast\amalg M}$ where $\ast\amalg M$ is the linear order given by adding a new smallest element $\ast$ to $M$ and $a_0\amalg e$ is the map defined by sending $\ast$ to $a_0$ and any element $m$ of $M$ to $e\br{m}$. Assuming that $\sigma$ is well-defined, this clearly gives a section, so it remains to check just that. In other words, we need to see that, if $\gamma\br{x}\leq g$, then $\gamma\br{\sigma\br{x}}\leq g$, so let us fix some $x$. For any edge in $\gamma\br{\sigma\br{x}}$ between vertices other than $a_0$, this graph coincides with $\gamma\br{x}$. Thus, it suffices to consider the edges between $a_0$ and some $a\in A\setminus\cb{a_0}$. If $a_0\overset{i}{\to}a$ in $\gamma\br{x}$, then the same is true in $\gamma\br{\sigma\br{x}}$. To see this, note that $a_0\to a$ in $\gamma\br{x}$ gives that the string obtained from $x$ by removing all symbols besides $a_0$ and $a$ begins with some string of $a_0$'s, the corresponding string for $\sigma\br{x}$ also begins with a (one symbol longer) string of $a_0$'s and the weight between the two elements is unchanged. If, on the other hand, $a\overset{i}{\to} a_0$ in $\gamma\br{x}$, then we have $a_0\overset{i+1}{\to}a$ in $\gamma\br{\sigma\br{x_0}}$ --- $a_0$ now occurs before $a$ in the string corresponding to $\sigma\br{x_0}$ and the symbols switch one additional time. Further, as $a_0$ is the smallest element of $g$, we must have $a_0\overset{j}{\to}a$ in $g$ for some $j>i$ and, so, $i+1\leq j$ and $\gamma\br{\sigma\br{x}}\leq g$, as well.
\end{proof}

\begin{theorem}
  The map $\gamma$ gives rise to an equivalence of operads between $\MS_m\br{A}$ and $\abs{\calg_m\br{A}}$.
\end{theorem}
\begin{proof}
  The space $\MS_m\br{A}$ is the totalisation of a geometric realisation of a multisimplicial set and each $\abs{\call_g\br{A;n}}$ is the totalisation of the realisation of a multisimplicial subset, hence $\gamma$ is sufficiently cofibrant and so this follows from Corollary~\ref{corEquivalence} using the previous proposition.
\end{proof}

\appendix
\begin{appendices}
  \section{Some (co)simplicial generalities}\label{appA}
  For the precise procedure of turning a multisimplicial cosimplicial set into a topological space, we will recall some formalities about tensors and cotensors between (co)simplicial objects.
  \begin{enumerate}[itemindent=*,leftmargin=0pt]
  \item If $X$ is a simplicial set and $Y$ a cosimplicial one, the set $X\otimes Y$ is the set of equivalence classes of pairs $\br{x,y}\in X_m\times Y_m$ where, for each morphism $\theta: \sbr{n}\to \sbr{m}$ in $\Delta$, $x\in X_m$ and $y\in Y_n$, one identifies the pairs $\br{\theta^\ast x,y}$ and $\br{x,\theta_\ast y}$. We will write $x\otimes y$ for the equivalence class of $\br{x,y}$ --- this allows the equivalence relation to be expressed in the familiar form $x\theta\otimes y=x\otimes \theta y$ where we write the simplicial and cosimplicial structures as right and left actions. The set $X\otimes Y$ is a coend and also written as $\int^n X_n\times Y^n$, $\int^\Delta X_n\times Y^n\text{d}\!n$, or $X\otimes_{n}Y$.

    More generally, if $X$ is a simplicial object and $Y$ is a cosimplicial object in a cocomplete monoidal category, we can construct $X\otimes Y$ as the coequaliser
    \[ \coprod_{f\colon m\to n\in \mathrm{Mor}\br{\Delta}}X_n\otimes Y^m\rightrightarrows \coprod_{m\in \Delta} X_m\otimes Y^m\to X\otimes Y \]
    where the first map is given by the simplicial action of $X$ to map $X_n\otimes Y^m$ to $X_{m}\otimes Y^{m}$ and the second map is given by the cosimplicial action of $Y$ to map $X_{n}\otimes Y^{m}$ to $X_{n}\otimes Y^{n}$.
  \item If $X$ is a simplicial set and $Y$ is a cosimplicial {\em space\/} then $X\otimes Y$ inherits its topology as a quotient of a sum of copies of the spaces $Y^n$. In particular, if $Y$ is the cosimplicial space of standard simplices $\abs{\Delta^\bullet}$, then $X\otimes Y=\abs{X}$ is the geometric realisation of $X$. More generally, one can think of $X\otimes Y$ as a geometric realisation of $X$ with respect to $Y$ and write $\abs{X}_Y$ for $X\otimes Y$ in this case.
  \item The tensor product is associative in the following sense: if $X$ is a simplicial set, $Y$is a simplicial cosimplicial set, and $Z$ is a cosimplicial set, then
    \[ \br{X\otimes Y}\otimes Z\cong X\otimes\br{Y\otimes Z} \]
    with the isomorphism natural in all three arguments. More explicitly, this means
    \[\int^p\br{\int^n X_n\times Y_p^n}\times Z^p\cong \int^{n}X_n\times \br{\int^{p}Y_p^n\times Z^p}\]
    For example, if $X$ is a simplicial set and $Y$ is a cosimplicial simplicial set, then the geometric realisation $\abs{Y}$ is a cosimplicial space and the realisation of $X$ with respect to $\abs{Y}$ is the same as the realisation of $X\otimes Y$, i.e.\ \[ \abs{X}_{\abs{Y}}=\abs{X\times Y} \]
    as the left hand side is $X\otimes \br{Y\otimes \abs{\Delta^\bullet}}$ and the right is $\br{X\otimes Y}\otimes \abs{\Delta^\bullet}$.
  \item In the case of multi-simplicial and multi-cosimplicial objects, one can integrate over each variable, separately. This means that, if $X$ is an $A$-fold simplicial set and $Y$ is a $B$-fold simplicial cosimplicial set, then, for any $a\in A$, the tensor
    \[ X\otimes_a Y=\int^{a}X_a\times Y^a \]
    is an $A\sbr{B/a}$-fold simplicial set.
  \item For example the substitution of lattice paths from Definition~\ref{defSubstitution} is a map $\call\br{A;n}\otimes_a\call\br{B;\bullet}\to \call\br{A\sbr{B/a};n}$. The fact that this is well-defined is simply the Remark~\ref{remSubstitution}.

    In fact, this map is even an isomorphism: as the substitution does not affect the part of the lattice path over $A$ outside of $a$, we may assume that $A=\cb{a}$, but in this case $\call\br{A;n}\cong\Delta\br{n}$ is the standard simplicial $n$-simplex and $\call\br{A\sbr{B/a};n}\cong\call\br{B;n}$ with the isomorphism induced by the substitution. Since $\Delta^\bullet$ is the unit for the tensor (by the Yoneda lemma), this just means that our map is given by $\call\br{A;n}\otimes_a\call\br{B;\bullet}=\Delta^n\otimes_a\call\br{B;\bullet}\cong\call\br{B;n}\isoto \call\br{A\sbr{B/a};n}$.
  \item If $X$ is a cosimplicial space, then its totalisation $\mathrm{Tot}\br{X}$ (for which we introduced the alternative notation $-X-$) is the space of natural transformations from the standard cosimplicial space $\abs{\Delta^\bullet}$ into $X$
    \[ \mathrm{Tot}\br{X}=\Top^\Delta\br{\abs{\Delta^\bullet},X^\bullet} \]
    It has a natural topology as a subspace of the product of the spaces $\br{X^n}^{\Delta^n}$ of maps $\Delta^n\to X^n$.
  \item There is also a more general way to interpret the totalisation construction: if we are given two cosimplicial spaces $X$ and $Y$, then the mapping space $\br{X^\bullet}^{Y^\bullet}$ is both simplicial and cosimplicial, allowing us to take the end over the category $\Delta$ to obtain a space $\int_n\br{X_n}^{Y^n}$ --- this happens to be the space of natural transformations $X^\bullet\to Y^\bullet$. In particular, if we choose $Y=\abs{\Delta^\bullet}$, this recovers the notion of the totalisation. Thus, we may think of totalising as a sort of dual construction to geometrically realising.
  \end{enumerate}
  \begin{lemma}
    If $X$ and $Y$ are simplicial cosimplicial spaces, then there is a natural map
    \[ \Tot{\abs{X}}\times \Tot{\abs{Y}}\to \Tot{\abs{X\otimes Y}} \]
  \end{lemma}
  \begin{proof}
    First, note that, for a simplicial cosimplicial space $X$, the geometric realisation $\abs{X^\bullet_\bullet}$ is again a cosimplicial space. On the other hand, $X\otimes Y=\br{\int^{n}X_n^p\times Y_m^n}_{p,m}$ is a cosimplicial simplicial space, so the space $\Tot{\abs{X\otimes Y}}$ is, indeed, well-defined. Now, given natural transformations $\varphi: \abs{\Delta^\bullet}\to \abs{X}$ and $\psi:\abs{\Delta^\bullet}\to\abs{Y}$, the latter defines a natural transformation of cosimplicial spaces
    \[ \psi_\ast: \abs{X^p}=X^p\otimes \abs{\Delta^\bullet}\to X^p\otimes \abs{Y} \]
    but the realisation $X^p\otimes \abs{Y}$ of $X$ with respect to $\abs{Y}$ is $\abs{X^p\otimes Y}$, as seen previously, so, in fact, $\psi_\ast: \abs{X}\to \abs{X\otimes Y}$. Composing this map with $\varphi$ gives a map $\abs{\Delta^\bullet}\to \abs{X\otimes Y}$. Thus, we obtain a natural map
    \begin{align*}
      &\Tot{\abs{X}}\times \Tot{\abs{Y}}\to \Tot{\abs{X\otimes Y}}\\
      &\br{\varphi,\psi}\mapsto \psi_\ast\circ\varphi
    \end{align*}
  \end{proof}
  
  \section{Some remarks on the literature}\label{appB}
  As mentioned in the introduction, the definition of the little $n$-cubes operads goes back to Boardman-Vogt~\cite{BV73} and May~\cite{May72}, who also introduced the notion of an $E_n$-operad. For $n=\infty$, Barratt and Eccles introduced their operad in~\cite{BE74}; this operad is essentially our operad $\Gamma$ given by the categories $\mathsf{Lin}$. The suboperads $\Gamma_n$ were introduced by Jeff Smith in~\cite{Smi89}, and conjectured to be $E_n$-operads in this paper.

 Kashiwabara~\cite{Kas93} was probably the first person to publish about this conjecture. He already used graphs, and wrote a directed graph on $A = \cb{1,\dots,p}$ with weights $1,\dots, n$ as a function $f: \mathrm{Ind}\br{p} \to \mathrm{J}\br{n}$ where $\mathrm{Ind}\br{p}$ is the set of two-element subsets of $A$ and $\mathrm{J}\br{n}$ is the product of $\cb{+,-}$ (for the directions in the graph) and $\cb{1,\dots,n}$ (for the weights). He uses the terms coherent for what we (following Berger) call complete, and non-degenerate for what we call proper. Furthermore, he writes $C_n\Sigma_p$ for the simplicial set $\Gamma_n\br{A}$ (where $A = \cb{1,\dots,p}$ again), and $\mathrm{C}\br{n,p}$ for what we call $\PP_n\br{A}$ (the proper graphs), cf. Definition 3.7 in loc.\ cit.. Using F. Cohen's description of the homology of configuration spaces~\cite{Coh76} Kashiwabara proceeds to prove a homology equivalence between (in our notation) the configuration space $\Conf_n\br{A}$ and the geometric realization $\abs{\PP_n\br{A}}$ of the proper graphs (cf.~\cite{Kas93}, 3.10), as well as a homotopy equivalence $\abs{\PP_n\br{A}} \xrightarrow{\raisebox{-0.2em}{$\sim$}}\abs{\Gamma_n\br{A}}$ (loc.\ cit.\ 6.1), and concludes in a somewhat indirect way that $\abs{\Gamma_n\br{A}}$ is homotopy equivalent to $\Conf_n\br{A}$ (loc.\ cit.\ 6.2).

 The decomposition of $\Conf_n$ into proper cells ($\psi^{-1}\br{g}$ in our Definition~\ref{defProper}) goes back to Fox and Neuwirth~\cite{FoN62} for $n=2$, and has been discussed extensively in the literature. For example, Getzler and Jones~\cite{GJ94} defined the cells in terms of nested preorders $\pi_1\leq\cdots\leq \pi_n$ on $A$, which are related to the proper graphs $g$ by defining $a\leq b$ in $\pi_i$ if and only if $a\overset{j}{\to} b$ in $g$ for some $j\leq i$. For further discussion of Fox-Neuwirth cells and applications, we refer to Giusti-Sinha~\cite{GS12}, Ayala-Hepworth~\cite{AH14}, and the references cited there.
 
  In a subsequent paper~\cite{Ber97}, Berger introduces the complete graph operad $\calg_n$, and shows that it is equivalent to $\Gamma_n$ as an operad, as stated above Theorem~\ref{thmBEisEn}. Another proof of this fact appears in~\cite{BFSV98}. Berger next argues that $\calc_n$ is equivalent to $\calg_n$ by proving an equivalence between $\Conf_n$ and $\PP_n$ (like~\cite{Kas93} does) and by arguing that $\PP_n$ is equivalent to $\calg_n$ by claiming that the inclusion of $\PP_n$ into $\calg_n$ has an adjoint, after quotienting by the action of $\Sigma_\bullet$ --- unfortunately, there was a slight mistake in the proof and such an adjoint cannot exist, as mentioned above. Moreover, Berger’s construction really lands in the larger operad $\calg_n^{\mathrm{ext}}$, as observed in Brun et.al.~\cite{BFV04}, who claim without further explanation that one can prove that $\calc_n$ is equivalent to $\calg_n^{\mathrm{ext}}$ by the same arguments as used in~\cite{Ber97}. A slightly more detailed version of Berger's arguments can also be found in~\cite{Ber96}.

  We hope to have established in a more rigorous way that $\PP_n$ is equivalent to $\calg_n$ as symmetric sequence of posets, and that $\calg_n$ and $\calg_n^{\mathrm{ext}}$ are indeed equivalent to each other and to $\calc_n$ as operads. Our argument that $\Gamma_n$ is an $E_n$-operad could therefore be based on a map of operads from $\Gamma_n$ to $\calg_n$ as we have seen.
  Our proof that $\calm_n$ is an $E_n$-operad largely follows~\cite{BFSV98}, as already mentioned.

  The fact that the Fulton-MacPherson operad $\FM_n$ is an $E_n$-operad was established by Salvatore, who constructed an equivalence of operads $W\calc_n \to \FM_n$, where $W$ denotes the Boardman-Vogt resolution. For us, having established that $\calg_n$ is an $E_n$-operad via a map $\Conf_n \to \calg_n$ (our Proposition~\ref{propGisEn}) the fact that $\FM_n$ is also equivalent to $\calg_n$ follows rather easily, as we have seen.
  Our proof that the McClure-Smith operad $\MS_n$ is an $E_n$-operad closely follows the same ideas as in~\cite{MS04}, although we hope to have clarified and simplified a few points. The observation that elements in $\call\br{A;n}$ can be pictured as paths in an $A$-dimensional lattice with $n$ ``stops'', i.e.\ as paths subdivided in $n+1$ segments, comes from~\cite{BB09}. They show that these lattice paths form a coloured operad from which the McClure-Smith operad can be constructed using an operation they call condensation.

\end{appendices}

\bibliographystyle{amsalpha}
\providecommand{\bysame}{\leavevmode\hbox to3em{\hrulefill}\thinspace}
\providecommand{\MR}{\relax\ifhmode\unskip\space\fi MR }
\providecommand{\MRhref}[2]{%
  \href{http://www.ams.org/mathscinet-getitem?mr=#1}{#2}
}
\providecommand{\href}[2]{#2}

\end{document}